\documentclass[12pt,reqno]{amsart}

\usepackage{amssymb,amsmath,graphicx,amsfonts,euscript}
\usepackage{color}
\usepackage{cite}

\allowdisplaybreaks

\newcommand{\Om}{\Omega}

\newcommand{\na}{\nabla}

\def\B{\tilde{b}}
\def\U{\tilde{u}}
\def\du{\tilde{U}}
\def\db{\tilde{B}}
\newcommand{\p}{\partial}
\newcommand{\C}{\cdot}

\def\A{\alpha}
\def\BB{\beta}
\def\CC{\gamma}

\def\va{\varphi}
\def\ps{\psi}

\def\De{\Delta}
\def\na{\nabla}

\def\v{\varepsilon}
\newcommand{\ddt}{\frac{d}{dt}}
\newcommand{\f}{\frac{1}{2}}

\newtheorem{thm}{Theorem}[section]
\newtheorem{prop}[thm]{Proposition}
\newtheorem{lem}[thm]{Lemma}

\newtheorem{re}[thm]{Remark}
\newtheorem{no}[thm]{Notation}
\newtheorem{de}[thm]{Definition}

\newcommand{\beq}{\begin{equation}}
\newcommand{\eeq}{\end{equation}}
\newcommand{\ben}{\begin{eqnarray}}
\newcommand{\een}{\end{eqnarray}}
\newcommand{\beno}{\begin{eqnarray*}}
\newcommand{\eeno}{\end{eqnarray*}}
\everymath{\displaystyle}

\makeatletter
\@namedef{subjclassname@2020}{%
	\textup{2020} Mathematics Subject Classification}
\makeatother

\numberwithin{equation}{section}

\begin{document}

\title[\resizebox{5in}{!}{The primitive equations with magnetic field approximation of the 3D MHD equations}]{ The primitive equations with magnetic field approximation of the 3D MHD equations}
\author[Lili Du and  Dan Li]{Lili Du$^{1}$ and  Dan Li$^{2}$}

\address{$^1$ Department of Mathematics, Sichuan University, Chengdu 610064, P.R. China}

\email{dulili@scu.edu.cn}

\address{$^2$ Department of Mathematics, Sichuan University, Chengdu 610064, P.R. China}

\email{dandy0219@hotmail.com}

\vskip .2in
\begin{abstract}
 In our earlier work \cite{DLL}, we have shown the global well-posedness of strong solutions to the three-dimensional primitive equations with the magnetic field (PEM) on a thin domain. The heart of this paper is to provide a rigorous justification of the derivation of the PEM as the small aspect ratio limit of the incompressible three-dimensional scaled magnetohydrodynamics (SMHD) equations in the anisotropic horizontal viscosity and magnetic field regime. For the case of $H^1$-initial data case, we prove that global Leray-Hopf weak solutions of the three-dimensional  SMHD equation strongly converge to the global strong solutions of the PEM. In the $H^2$-initial data case, the strong solution   of the SMHD can be extended to be a global one for small $\v$. As a consequence, we observe that the global strong solutions of the SMHD strong converge to the global strong solutions of the PEM. As a byproduct, the convergence rate is of the same order as the aspect ratio parameter.
\end{abstract}
\maketitle

\section{Introduction}
\subsection{Background and motivation}
\vskip.1in

The magnetohydrodynamics (MHD) system studies the dynamics of electrically conducting fluids under the influence of magnetic fields. There are many examples of conducting fluids, including
plasmas, liquid metals, electrolytes, etc. The main idea of magnetohydrodynamics is that conducting fluids can support magnetic fields. More precisely, magnetic fields can induce currents in a moving conducting fluid,  creating forces on the fluid and also changing the magnetic fields themselves. The subject of magnetohydrodynamics unites classical fluid dynamics with electrodynamics, and references can be found in \cite{HA1, DB,  HB, PAD, BD1}. Besides their wide physical applications, the global well-posedness of the MHD equations is an active topic in mathematics. The existence and uniqueness results for weak and strong solutions of the 2D MHD equations are well known by Duvaut and Lions in \cite{DL}. In general, for the 3D case, it is currently unknown whether the solutions  can develop finite-time singularities even if the initial value is sufficiently smooth. Different criteria for regularity in terms of the velocity field, the magnetic field, the pressure, or their derivatives have been proposed (see \cite{CC, KY, HC, CZ, QC1, AH, HP, JW1, JW2, JW3, JW4, JW, WY1} and references therein). One of the most elegant works is given by He and Xin in \cite{HC, CZ}, in which they first realized that the velocity fields played a dominant role in the regularity of the solution to 3D incompressible MHD equations.


\vskip .1in
 Without the magnetic field, the MHD equations reduce to the Navier-Stokes equations. In the context of the geophysical flow concerning the large-scale oceanic dynamics, the vertical scale of the global atmosphere is much smaller than the horizontal one. Based on this relationship, by scaling the incompressible Navier-Stokes equations concerning the aspect ratio parameter and taking the small aspect ratio limit, one formally obtains the primitive equations for the large-scale oceanic dynamics. The global existence of strong solutions for the 2D case was established by Bresch et al. \cite{DB1} and Temam and Ziane in \cite{RT1}, while the 3D primitive equations have already been known since the breakthrough work by Cao and Titi  \cite{CCET1}, see also   \cite{GK1, IK1, MH1, MH2}. In the last few years, developments concerning the global well-posedness to the anisotropic primitive equations were also made, see \cite{CCL1, CCL2, CCL3, CCL4, CCL5, CCL6}. The rigorous justification of the primitive equations from the scaled Navier-Stokes equations was  studied by Az\'erad and Guill\'en in \cite{AF}. By relying on the result in \cite{AF} to prove the weak convergence, Li and Titi in \cite{JKL} showed that the Navier-Stokes equations strongly converge to the primitive equations. Furukawa et al. \cite{KF1} extended the results by  Li and Titi, and the convergence result is shown in the much more general  $L^p-L^q$ setting.  Furukawa, Giga and Kashiwabara \cite{KF2} showed that the solution to the scaled Navier-Stokes equations with Besov initial data converges to the solution to the primitive equation with the same initial data.

\vskip .1in

In our paper \cite{DLL}, the  anisotropic viscosity and magnetic diffusivity scaling in the horizontal and vertical directions, so that  the PEM derived from the scaled MHD equation, as the aspect ratio  goes to zero.  Moreover, we have shown the global well-posedness of strong solutions to the three-dimensional incompressible PEM without any small assumption on the initial data. Motivated by the rigorous justification  of the limiting process, the convergence from Navier-Stokes equations to primitive equations by Li and Titi in \cite{JKL} in the strong setting. In this paper, we will investigate  the  rigorously justify the scaled MHD equations convergence strongly to the PEM, globally and uniformly in time, and that the convergence rate is of the same order as the aspect ratio parameter. These will  be the main consequences of the studies in this paper.

\subsection{Anisotropic MHD equations}
\vskip.1in
In this subsection we consider the three-dimensional anisotropic MHD equations on an $\varepsilon$-dependent thin domain
 $$\Omega_\varepsilon:=M\times(-\varepsilon,\varepsilon)\subset \mathbb{R}^3,$$
 where $\varepsilon>0$ is a very small parameter, and $M=(0,L_1)\times(0,L_2)$, for two positive constants $L_1$ and $L_2$ of order $O(1)$ with respect to $\varepsilon$.

 The incompressible three-dimensional anisotropic MHD system is
\begin{equation}\label{1}
 \left\{
\begin{array}{l}
\p_t u + u\cdot \nabla u + \nabla p - \mu\Delta_H u-\nu\partial_z^2 u=b\cdot\nabla b,\\
\p_t b+ u\cdot \nabla b-\kappa\Delta_H b-\sigma\partial_z^2 b=b\cdot\nabla u, \\
\nabla\cdot u =0,  \quad \nabla\cdot b=0,
 \end{array} \right.
 \end{equation}
with\\
\leftline{$- u=(\U, u_3)$ \text{is the velocity field},~~ \text{where}~~$\U=(u_1, u_2)$~~\text{is the horizontal velocity field},}
\leftline{$- b=(\B, b_3)$\text{ is the magnetic field},~~\text{where}~~ $\B=(b_1, b_2)$~~\text{is the horizontal magnetic field},}
  \leftline{$- p$ \text{is the pressure},}
 \leftline{$- \mu$ \text{is the horizontal viscous coefficient},}
 \leftline{$- \nu$ \text{ is the vertical viscous coefficient},}
  \leftline{$- \kappa$ \text{is the horizontal magnetic diffusivity coefficient},}
 \leftline{$- \sigma$ \text{is the vertical magnetic diffusivity coefficient},}
\leftline{$- \De=\De_{H}+\p_z^2$,~~ \text{where}~~$\De_{H}=\p_x^2+\p_y^2$~~\text{is the horizontal Laplacian},}
\leftline{$- \na=(\p_x,\p_y,\p_z)$ \text{is gradient operator},}
\leftline{$-\na\C$ \text{is divergence operator}.}

  Similar to the case considered in Az\'erad-Guill\'en \cite{AF}, it is  emphasized that the anisotropic viscosity hypothesis is fundamental for the derivation of the primitive equations. In this paper,  we suppose that  $\mu$ and $\nu$ have different orders, that is, $\mu=O(1)$ and $\nu=O(\varepsilon^2)$. The orders of magnetic diffusivity coefficients $\kappa$ and $\sigma$ are similar to the viscous coefficients. For the sake of simplicity, we set  $\mu=1$ and $\nu=\varepsilon^2,$ similarly, $\kappa=1$ and $\sigma=\varepsilon^2.$
\vskip .1in

\subsection{Scaled  MHD equations}
\vskip .1in

 We carry out the following scaling transformation to  the  MHD equations (\ref{1}) such that the resulting system is defined on a fixed domain independent of $\varepsilon$. To this end, we introduce the new unknowns,
$$u_\varepsilon=(\U_\varepsilon, u_{3,\varepsilon}),\quad b_\varepsilon=(\tilde{b}_\varepsilon, b_{3,\varepsilon}),\quad p_\varepsilon(x,y,z,t)=p(x,y,\varepsilon z,t),$$
$$\U_\varepsilon(x,y,z,t)=\U(x,y,\varepsilon z,t)=(u_1(x,y,\varepsilon z,t), u_2(x,y,\varepsilon z,t)),$$
$$\tilde{b}_\varepsilon(x,y,z,t)=\tilde{b}(x,y,\varepsilon z,t)=(b_1(x,y,\varepsilon z,t), b_2(x,y,\varepsilon z,t)),$$
and
$$u_{3,\varepsilon}(x,y,z,t)=\frac{1}{\varepsilon}u_3(x,y,\varepsilon z,t),\qquad b_{3,\varepsilon}(x,y,z,t)=\frac{1}{\varepsilon}b_3(x,y,\varepsilon z,t).$$

For any $(x,y,z)\in\Omega: = M\times (-1,1) $ and  $t\in [0,\infty)$, then  $u_\varepsilon$, $b_\varepsilon$ and $p_\varepsilon$ satisfy the following incompressible  scaled  MHD equations (SMHD)
\begin{equation}\label{2}
 \left\{
\begin{array}{l}
\p_t \U_\varepsilon + u_\varepsilon\cdot \nabla \U_\varepsilon + \nabla_H p_\varepsilon - b_\varepsilon\cdot\nabla\tilde{b_\varepsilon}-\Delta\U_\varepsilon=0,\\
\varepsilon^2(\partial_t u_{3,\varepsilon}+u_\varepsilon\cdot\nabla u_{3,\varepsilon}-\Delta u_{3,\varepsilon} - b_\varepsilon\cdot\nabla b_{3,\varepsilon})+\partial_z p_\varepsilon=0,\\
\partial_t\tilde{b_\varepsilon}+u_\varepsilon\cdot\nabla\tilde{b_\varepsilon}-\Delta\tilde{b_\varepsilon}-b_{\varepsilon}\cdot\nabla \U_{\varepsilon}=0, \\
\varepsilon^2(\partial_t b_{3,\varepsilon}+u_{\varepsilon}\cdot\nabla b_{3,\varepsilon}-\Delta b_{3,\varepsilon}-b_\varepsilon\cdot\nabla u_{3,\varepsilon})=0,\\
\nabla\C u_\varepsilon=0,\quad \nabla\C b_\varepsilon=0,
 \end{array} \right.
 \end{equation}
with the following initial conditions
\begin{equation}\label{pm1}
 \left\{
\begin{array}{l}
(\U_\varepsilon,u_{3,\varepsilon})|_{t=0}=(\U_{\varepsilon,0},u_{3,\varepsilon,0}),\\ (\tilde{b}_\varepsilon,b_{3,\varepsilon})|_{t=0}=(\B_{\varepsilon,0},b_{3,\varepsilon,0}).\\
 \end{array} \right.
 \end{equation}

 The above equations (\ref{2}) are defined in the fixed domain $\Omega$. Throughout this paper, we set $\na_H$ to denote $(\p_x, \p_y)$. In addition, we equip the system (\ref{2}) with the following periodic boundary conditions,
\begin{align}\label{pm2}
\U_{\v}(x,y,z-1,t)=\U_{\v}(x+L_1,y+L_2,z+1,t),
\end{align}
\begin{align}\label{pm2}
 u_{3,\v}(x,y,z-1,t)=u_{3,\v}(x+L_1,y+L_2,z+1,t),
\end{align}
\begin{align}\label{pm4}
\B_{\v}(x,y,z-1,t)=\B_{\v}(x+L_1,y+L_2,z+1,t),
\end{align}
\begin{align}\label{pm5}
b_{3,\v}(x,y,z-1,t)=b_{3,\v}(x+L_1,y+L_2,z+1,t),
\end{align}
and
\begin{align}\label{6}
&p_{\v}(x,y,z-1,t)=p_{\v}(x+L_1,y+L_2,z+1,t).
\end{align}

Furthermore,  for simplicity, we suppose that the space of periodic functions with  respect to $z$ with the following symmetry
\begin{align}\label{b2}
&\U_\v(x,y,z,t)=\U_\v(x,y,-z,t),\quad u_{3,\v}(x,y,z,t)=-u_{3,\v}(x,y,-z,t),\\
&p_{\v}(x,y,z,t)=-p_{\v}(x,y,-z,t),\quad \B_{\v}(x,y,z,t)=\B_{\v}(x,y,-z,t),
\end{align}
and
\begin{align}\label{b3}
\quad b_{3,\v}(x,y,z,t)=-b_{3,\v}(x,y,-z,t).
\end{align}

Note that the dynamics of SMHD preserve these symmetry conditions. In other words, they are automatically satisfied as long as they are satisfied initially. So in this article, without further mention, we always assume that the initial horizontal velocity and  magnetic field $(\U_{0,\v}$, $\tilde{b}_{0,\v})$ satisfy that
\begin{align*}\label{b4}
\U_{0,\v}~,\B_{0,\v}~ \text{are  periodic in}~x,~y,~z,~\text{and are even in}~z.
\end{align*}

By the classic theory see, e.g., \cite{DL}, for any initial data $(u_0, b_0)\in L^2(\Om)$, there is global weak solution $(u, b)$ to the SMHD equation (\ref{2}), subject to (\ref{pm1})-(\ref{6}), here the weak solutions is defined as follows.

\vskip.1in
\begin{de}\label{7}
A weak solution  $(u,b)$ of the SMHD (\ref{2}) is called Leray-Hopf weak solution, if $(u,b)\in C_w([0,\infty);L_\sigma^2(\Om))\cap L_{loc}^2([0,\infty),H^1(\Om))$, where the subscript $w$ means weakly continuous and $L_\sigma^2(\Om)$ denotes the space consisting of all divergence-free functions in $L^2(\Om)$ and
\begin{align}
&\|\U(t)\|_2^2+\|\B(t)\|_2^2+\v^2\|u_3(t)\|_2^2+\v^2\|b_3(t)\|_2^2\notag\\
&+2\int_0^t(\|\na \U\|_2^2+\|\na \B\|_2^2+\v^2\|\na u_3\|_2^2
+\v^2\|\na b_3\|_2^2)\,ds\notag\\
\leq&\|\U_0\|_2^2+\|\B_0\|_2^2+\v^2\|u_{3,0}\|_2^2+\v^2\|b_{3,0}\|_2^2,
\end{align}
for a.e. $t\in[0,\infty).$

Moreover, the following integral identity holds
\begin{align*}
&\int_{Q}\Big[-(\U\C\p_t\varphi_H+\v^2u_{3}\p_t\varphi_3)-(\B\C\p_t\psi_H+\v^2b_{3}\p_t\psi_3)
+(u\C\na)\U\C\varphi_H\\
&-(b\C\na)\B\C\varphi_H+\v^2u\C\na u_{3}\varphi_3-\v^2b\C\na b_{3}\varphi_3+(u\C\na)\B\C\psi_H-(b\C\na)\U\C\psi_H\\
&+\v^2u\C\na b_{3}\psi_3-\v^2b\C\na u_{3}\psi_3+\na\U:\na\varphi_H+\v^2\na u_{3}\C\na\varphi_3+\na\B:\na\psi_H\\
&+\v^2\na b_{3}\C\na\psi_3\Big]\,dxdydzdt\\
&=\int_{\Om}(\U_0\C\varphi_H(\cdot,0)+\v^2u_{3,0}\varphi_3(\C,0))\,dxdydz
+\int_{\Om}(\B_0\C\psi_H(\C,0)+\v^2b_{3,0}\psi_3(\C,0))\,dxdydz,
\end{align*}
 where $Q:=\Om\times(0,\infty)$, for any space periodic functions $\varphi=(\varphi_H,\va_3)$ and  $\ps=(\ps_H,\ps_3)$, with $\va_H=(\va_1,\va_2)$ and $\ps_H=(\ps_1,\ps_2)$, the divergence-free test functions   $\va$ and $\ps$ satisfy $\va\in C_0^{\infty}(\overline{\Om}\times[0,\infty))$ and  $\ps\in C_0^\infty(\overline{\Om}\times[0,\infty))$.
\end{de}

\subsection{Primitive equations with magnetic field}
\vskip .1in

By taking the limit as  $\varepsilon\rightarrow 0$ in the SMHD (\ref{2}),  it is straightforward to obtain the following primitive equations with magnetic field (PEM)
\begin{equation}\label{b5}
 \left\{
\begin{array}{l}
\partial_t \U+u\cdot\nabla \U-\Delta \U-b\cdot\nabla\tilde{b}+\nabla_{H} p = 0,\\
\partial_z p=0,\\
\partial_t\tilde{b}+u\cdot\nabla\tilde{b}-\Delta\tilde{b}-b\cdot\nabla \U=0,\\
\nabla_{H}\cdot \U+\partial_z u_3=0, \qquad \nabla_{H}\cdot\tilde{b}+\partial_z b_3=0.
\end{array} \right.
 \end{equation}

Recalling that we consider the periodic initial-boundary value problem to the SMHD equations (\ref{2}), it is clear that one should impose the same boundary conditions and symmetry conditions on the corresponding limiting system (\ref{b5}). However, one only needs to impose the initial condition on the horizontal velocity field and magnetic field. Since $u_{3,0}$ and $b_{3,0}$ are odd in $z$, we have $u_{3,0}(x,y,0)=b_{3,0}(x,y,0)=0$. Then, $(u_{3,0}, b_{3,0})$ can be determined uniquely  by the incompressibility conditions as
\begin{align}\label{3}
u_{3,0}(x,y,z)=-\int_0^z\nabla_H\cdot \U_0(x,y,\xi)\,d\xi,
\end{align}
and
\begin{align}\label{30}
b_{3,0}(x,y,z)=-\int_0^z\nabla_H\C\B_0(x,y,\xi)\,d\xi.
\end{align}

Similarly, $(u_3,b_3)$ can also be determined uniquely  by the incompressibility conditions as
\begin{align}\label{t3}
u_{3}(x,y,z,t)=-\int_0^z\nabla_H\cdot \U(x,y,\xi,t)\,d\xi,
\end{align}
and
\begin{align}\label{mm}
b_{3}(x,y,z,t)=-\int_0^z\nabla_H\cdot\tilde{b}(x,y,\xi,t)\,d\xi.
\end{align}

Due to these facts concerning the solutions to (\ref{b5}), we only solve the horizontal components $(\U, \B)$, and the vertical components $(u_3, b_3)$ are uniquely determined by (\ref{t3}) and (\ref{mm}). Throughout this paper, all the velocities and the magnetic field encountered in this paper are of average zero. We will not review this fact in  the rest of this article, before using the Poincar\'e inequality.



\subsection{Main ideas of the construction}
\vskip.1in
The main results in this paper are concerned with the strong convergence from the SMHD equations to the PEM, as the aspect ratio parameter goes to zero. For the first result, given $(\U_0, \B_0)\in H^1(\Om)$, we prove that global Leray-Hopf weak solutions of the three-dimensional  SMHD equation strongly converge with the global strong solutions of the PEM.  More precisely, we prove the strong convergence
 $$(\U_{\v}, \v u_{3,\v}, \B_{\v}, \v b_{3, \v})\rightarrow(\U,0,\B,0)~~ \text{in}~~ L^\infty(0,\infty;L^2(\Om)).$$
On the other hand, given the initial data $(\U_0,\B_0)\in H^2(\Om)$, the strong solution  $(\U_{\v},\v u_{3,\v},\B_{\v}, \v b_{3,\v})$ of the SMHD can be extended to be a global one for small $\v$. As a consequence, we observe that the global strong solutions of SMHD strong convergence to the global strong solutions of the PEM, that is
 $$(\U_{\v},\v u_{3,\v},\B_{\v}, \v b_{3,\v})\rightarrow (\U, 0,\B,0) ~~\text{in}~~ L^\infty(0,\infty; H^1(\Om)),$$
 and the converge rate of two regimes are  the order $O(\v)$.

We now make some comments on the analysis of this paper. The treatments on the estimates of the difference function $(U_{\v}, B_{\v})=(u_{\v}-u, b_{\v}-b)$ are different in the proofs of the first  and second results. For the case of the first result, since $(\U_{\v}, u_{3,\v}, \B_{\v}, b_{3,\v})$  is the Leray-Hopf weak solution, the energy estimates cannot be directly used for the system of difference between the SMHD and  the PEM, as it is usually used for strong solutions. Instead, we adopt a similar approach by Serrin in \cite{JS} (see also Bardos et al. \cite{CBM} and the reference therein), showing the weak-strong uniqueness of the Navier-Stokes equations. The difference is that in our article, the  ``strong solutions" is  played by the solutions of PEM, while the  ``weak solutions" is  played by the solutions of SMHD. Now, we explain in more detail  how we perform the global Leray-Hopf weak solutions  of SMHD strong convergence to the global strong solutions of the PEM,

(1) using $(\U,u_3, \B,b_3)$ as the testing functions for the SMHD, we get  one equality,

 (2) testing the PEM by $(\U_{\v},\B_{\v})$, it follows from integration by parts, we have one equality,

 (3) applying  the PEM by $(\U, \B)$, we obtain the basic energy identity of the PEM,

(4) Recalling the definition of Leray-Hopf weak solutions to the SMHD, we have the energy inequality of the SMHD.

Adding $(3)$ and $(4)$, then subtracting from  $(1)$ and $(2)$, we obtain a new integral inequality.
Appropriately manipulating this formulas, we get the desired \emph{a priori} estimates for $(U_{\v}, B_{\v})$.
\vskip.1in

For the second  case of  strong convergence, one gets the desired global in time estimates on $(U_{\v}, B_{\v})$ by using standard energy estimates. $(U_{\v}, B_{\v})$ denotes the difference between $u_{\v}$, $u$, $b_{\v}$ and $b$, namely,
$$U_{\v}=(\du_{\v},U_{3,\v}),\qquad \du_{\v}=\U_{\v}-\U,\qquad U_{3,\v}=u_{3,\v}-u_3,$$
and
$$B_{\v}=(\db_{\v},B_{3,\v}),\qquad \db_{\v}=\B_{\v}-\B,\qquad B_{3,\v}=b_{3,\v}-b_3,$$
then, one can easily verify that $U_{\v}=(\du_{\v},U_{3,\v})$ and $B_{\v}=(\db_{\v},B_{3,\v})$ satisfy the following system
\begin{equation}\label{n1}
 \left\{
\begin{array}{l}
\p_t \du_{\v}-(U_{\v}\C\na)\du_{\v}-\De\du_{\v}+\na_H P_{\v}+(u\C\na)\du_{\v}+(U_{\v}\C\na)\U\\
-B_{\v}\C\na \db_{\v}-(b\C\na)\db_{\v}-B_{\v}\C\na\B=0,\\
\v^2(\p_t U_{3,\v}+U_{\v}\C\na U_{3,\v}-\De U_{3,\v}+U_{\v}\C\na u_3+u\C\na U_{3,\v}
-B_{\v}\C\na B_{3,\v}\\
-b\C\na B_{3,\v}-B_{\v}\C\na b_3)+\p_z p_{\v}
=-\v^2(\p_t u_3+u\C\na u_3-\De u_3-b\C\na b_3),\\
\p_t \db_{\v}+U_{\v}\C\na \db_{\v}-\De\db_{\v}+u\C\na\db_{\v}+U_{\v}\C\na\B-B_{\v}\C\na \du_{\v}-b\C\na\du_{\v}-B_{\v}\C\na\U=0,\\
\v^2(\p_t B_{3,\v}+U_{\v}\C\na B_{3,\v}+u\C\na B_{3,\v}+U_{\v}\C\na b_3-\De B_{3,\v}-B_{\v}\C\na U_{3,\v}\\
-b\C\na U_{3,\v}-B_{\v}\C\na u_3)=-\v^2(\p_t b_3+u\C\na b_3-\De b_3-b\C\na u_3),\\
\na_H\C\du_{\v}+\p_z U_{3,\v}=0,\qquad \na_H\C\db_{\v}+\p_zB_{3,\v}=0.
\end{array} \right.
 \end{equation}

We  recommend some new thoughts described in the following process. At first, since the initial value of $(\tilde U_{\v},U_3, \tilde B_{\v}, B_3)$ disappears, and there is a small coefficient $\v^2$ in the front of the ``outside forcing'' terms on the right-hand side of the $(\ref{n1})_2$ and $(\ref{n1})_4$, we can perform the energy approach and take advantage of the small parameters  to obtain the required \emph{a priori} estimate on $(\tilde U_{\v}, U_3, \tilde B_{\v},B_3)$. Moreover,  the strong solution $(\U_{\v}, u_{3,\v}, \B_{\v},b_{3,\v})$ of the SMHD can be developed to the global solution, for small $\v$.
 Second, the  information of $(U_{3,\v},B_{3,\v})$ that  comes from equations $(\ref{n1})_2$ and $(\ref{n1})_4$ is always related to the parameter $\v$, which will ultimately go to zero.
 That is to say, the equations $(\ref{n1})_2$ and $(\ref{n1})_4$  have not provided  the information of  $\v$-independent of $(U_{3,\v},B_{3,\v})$. After the attainment of the desired \emph{a priori} estimates, the strong convergence follows instantly.


In our previous work \cite{DLL}, we showed the global existence of the strong solution and uniqueness (regularity) to the  three-dimensional incompressible PEM without any small assumption on the initial data. More precisely, there exists a unique strong solution globally in time for any given $H^2$-smooth initial data. As mentioned in the comments, the global well-posedness of strong solutions to the PEM plays a fundamental role proving  the strong convergence of the small aspect ratio limit of the SMHD to the PEM. The main results of \cite{DLL} are stated as follows.
\begin{prop}\label{5}
(see \cite[Remark 1.2]{DLL}).
If the initial data  $(\U_0,\B_0)$ belong to  $H^1(\Om)$,  then there exists  a unique global strong  solution to the  PEM (\ref{b5}), which satisfies $(\B,\U)\in L^\infty([0,\infty);H^1(\Om))\cap L^2([0,\infty);H^2(\Om)), (\p_t\B,\p_t\U)\in L^2([0,\infty);L^2(\Om)).$
\end{prop}

\begin{prop}\label{t1}
(see \cite[Theorem 1.1]{DLL}).
 Suppose that $(\U_0,\B_0) \in H^2(\Omega)$,  then there exists a unique  global strong solution $(\U,\B)\in L^{\infty}([0,\infty);H^2(\Om))\cap L^2([0,\infty);H^3(\Om))$ of the  PEM (\ref{b5}), subject to  the boundary and initial conditions (\ref{pm1})-(\ref{b3}). Moreover, we have the following estimate,
\begin{align*}
&\sup\limits_{0\leq t< \infty} \|\U\|_{H^2}^2(t)+\sup\limits_{0\leq t<\infty}\|\tilde{b}\|_{H^2}^2(t)+\int_0^\infty(\|\na \U\|_{H^2}^2+\|\partial_t \U\|_{H^1}^2)\,dt\\
&+\int_0^\infty(\|\na\tilde{b}\|_{H^2}^2+\|\partial_t\tilde{b}\|_{H^1}^2)\,dt
\leq C,
\end{align*}
for a constant $C$ depending only on $\|\U_0\|_{H^2}$, $\|\B_0\|_{H^2}$, $L_1$ and $L_2$.
\end{prop}

\vskip .1in
\subsection{The structure of this paper}

The remainder of this paper is organized as follows: Section 2 is dedicated to the basic notations and some Ladyzhenskaya-type inequalities. Section 3, this section is devoted to the case of $(\U_0,\B_0)\in H^1(\Om)$, we prove that global Leray-Hopf weak solutions of the three-dimensional  SMHD equations strongly converge to the global strong solutions of the PEM. In Section 4,  for the $(\U_0,\B_0)\in H^2(\Om)$ case, the strong solution of the SMHD can be extended to be a global one, for small $\v$. Moreover,  we observe that the global strong solutions of  the SMHD strong converge to the global strong solution of the PEM, and the convergence rate is  the same order as the aspect ratio parameter.

\section{Preliminaries}
In this section, we introduce the notations and some Ladyzhenskaya type inequalities  for some kinds of three dimensional integrals, which will be frequently used in the rest of this paper.

\begin{no}
For $q\in[1,\infty]$, we will denote the Lebesgue spaces on the domain $\Omega$ by $L^q=L^q(\Omega)$. For simplicity of notation we will use  $\|\cdot\|_q$  and $\|\cdot\|_{q,M}$ instead of $L^q(\Omega)$ and  $L^q(M)$. For $s\in\mathbb{N}$ the space $H^s(\Omega)$ consists of $f\in L^2(\Omega)$ such that $\nabla^\alpha f\in L^2(\Omega)$ for $|\alpha|\leq s$ endowed with the norm
$$\|f\|_{H^s(\Omega)}=\Big(\sum_{|\alpha|\leq s}\|\na^\alpha f\|_{L^2(\Omega)}^2\Big)^{\frac{1}{2}}.$$
\end{no}


Next,  we state some Ladyzhenskaya-type inequalities for some kinds of three dimensional integrals.
\begin{lem} \label{b7}
(see \cite[Lemma 2.1]{CCET} ). The following inequalities hold true
\begin{align*}
&\int_M\Big(\int_{-1}^1\alpha(x,y,z)\,dz\Big)\Big(\int_{-1}^1\beta(x,y,z)\gamma(x,y,z)\,dz\Big)dx dy \\
\leq & C\|\A\|_2^{\frac{1}{2}}\Big(\|\A\|_2^{\frac{1}{2}}+\|\nabla_{H}\A\|_2^{\frac{1}{2}}\Big)\|\BB\|_2\|\CC\|_2^{\frac{1}{2}}
\Big(\|\CC\|_2^{\frac{1}{2}}+\|\nabla_{H}\CC\|_2^{\frac{1}{2}}\Big),
\end{align*}
and
\begin{align*}
&\int_M\Big(\int_{-1}^1\A(x,y,z)\,dz\Big)\Big(\int_{-1}^1\BB(x,y,z)\CC(x,y,z)\,dz\Big)dx dy \\
\leq & C\|\A\|_2\|\BB\|_2^{\frac{1}{2}}\Big(\|\BB\|_2^{\frac{1}{2}}+\|\nabla_H \BB\|_2^{\frac{1}{2}}\Big)\|\CC\|_2^{\frac{1}{2}}\Big(\|\CC\|_2^{\frac{1}{2}}+\|\nabla_H \CC\|_2^{\frac{1}{2}}\Big),
\end{align*}
for any $\A$, $\BB$, $\CC$ such that the right-hand sides make sense and are finite, where $C$ is a positive constant depending only on $L_1$ and $L_2$.
\end{lem}
\begin{lem} \label{b8}
(see \cite[Lemma 2.2]{JKL}). Let $\varphi=(\varphi_1,\varphi_2,\varphi_3)$, $\phi$ and $\psi$ be periodic functions with  domain $\Omega$. Assume that $\varphi\in H^1(\Omega)$, with $\nabla\cdot\varphi=0$ in $\Omega$ ,$\int_{\Omega}\varphi dx dy dz =0$, and $\varphi_3|_{z=0}=0$, $\nabla\phi\in H^1(\Omega)$ and $\psi\in L^2(\Omega)$. $\varphi_H=(\varphi_1,\varphi_2)$  denotes the horizontal components of the function $\varphi$. Then,  the following estimate holds
\begin{align*}
\Big|\int_{\Omega}(\varphi\cdot\nabla\phi)\psi\,dxdydz\Big|\leq C\|\nabla\varphi_H\|_2^{\frac{1}{2}}\|\Delta\varphi_H\|_2^{\frac{1}{2}}\|\nabla\phi\|_2^{\frac{1}{2}}\|\Delta\phi\|_2^{\frac{1}{2}}\|\psi\|_2,
\end{align*}
where $C$ is a positive constant depending only on $L_1$ and  $L_2$.
\end{lem}

\vskip .3in
\section{Strong convergence I: the $H^1$ initial data case}
\vskip .1in
This section is devoted  to the strong convergence of the SMHD  to the PEM with initial data $(\U_0, \B_0) \in H^1(\Om)$. The main results are stated as follows.
\begin{thm}\label{t0}
Given a periodic function $(\U_0,\B_0)\in H^1(\Om)$, such that
$$\na_H\C\Big(\int_{-1}^1\U_0(x,y,z)\,dz\Big)=0, \qquad \int_{\Om}\U_0(x,y,z)\,dxdydz=0,$$
and
$$\na_H\C\Big(\int_{-1}^1\B_0(x,y,z)\,dz\Big)=0, \qquad \int_{\Om}\B_0(x,y,z)\,dxdydz=0.$$

Let $(\U_\v, u_{3,\v},\B_{\v}, b_{3,\v})$ be an arbitrary Leray-Hopf weak solution to the SMHD, $(\U, u_3, \B, b_3)$ be the unique global strong solution to the PEM, subject to (\ref{pm1})-(\ref{b3}). We denote
$$(\du_\v, U_{3,\v})=(\U_{\v}-\U, u_{3,\v}-u_3),\qquad (\db_{\v}, B_{3,\v})=(\B_\v-\B, b_{3,\v}-b_3 ).$$
Therefore,  the following a priori estimate holds
\begin{align*}
&\sup_{0\leq t<\infty}(\|\du_{\v}\|_2^2+\v^2\|U_{3,\v}\|_2^2+\|\db_{\v}\|_2^2+\v^2\|B_{3,\v}\|_2^2)(t)\\
&+\int_0^\infty (\|\na \du_{\v}\|_2^2+\v^2\|\na U_{3,\v}\|_2^2+\|\na \db\|_2^2+\v^2\|\na B_{3,\v}\|_2^2)\,ds\\
\leq&C\v^2(\|\U_0\|_2^2+\|\B_0\|_2^2+\v^2\|u_{3,0}\|_2^2+\v^2\|b_{3,0}\|_2^2+1)^2,
\end{align*}
for any $\v>0$, where $C$ is a positive constant depending only on $\|\U_0\|_{H^1}$, $\|\B_0\|_{H^1}$, $L_1$ and $L_2$. As a direct consequence, we get
$$(\U_{\v}, \v u_{3,\v}, \B_{\v}, \v b_{3,\v})\rightarrow (\U, 0, \B, 0),\quad \text{in}~L^\infty(0,\infty; L^2(\Om)),$$
$$(\na\U_{\v},\v\na u_{3,\v}, u_{3,\v}, \na\B_{\v}, \v\na b_{3,\v}, b_{3,\v})\rightarrow(\na\U, 0, u_3, \na\B, 0, b_3),\quad \text{in}~ L^2(0,\infty; L^2(\Om)),$$
and the convergence rate is of the order $O(\v)$.
\end{thm}

\begin{re}
The assumptions $\int_{\Om}\U_0\,dxdydz=0$ and $\int_{\Om}\B_0\,dxdydz=0$ are imposed only for the simplicity of the proof, and the same result still holds for the general case. One can follow the proof  presented in this paper, and establish the relevant a priori estimates on $(\U_{\v}-\overline \U_{0,\Om})$, $(\B_{\v}-\overline\B_{0,\Om})$, $(\U-\overline \U_{0,\Om})$ and  $(\B-\overline\B_{0,\Om})$, instead of  $(\U_{\v}, \B_{\v},\U, \B)$ themselves, where $\overline \U_{0,\Om}=\int_{\Om}\U_0\,dxdydz$ and  $\overline\B_{0,\Om}=\int_{\Om}\B_0\,dxdydz.$
\end{re}

This part is devoted  to the strong convergence of the SMHD  to the PEM with initial data $(\U_0, \B_0) \in H^1(\Om)$, in other words, we give the proof of Theorem \ref{t0}, let the initial data $(\U_0,\B_0)\in H^1(\Om)$, and assume
\begin{align}\label{s1}
\na_H\C\Big(\int_{-1}^1\U_0(x,y,z)\,dz\Big)=0 \quad \text{and} \quad \na_H\C\Big(\int_{-1}^1\B_0(x,y,z)\,dz\Big)=0,
\end{align}
for all $(x,y)\in M$. Using Proposition  \ref{5}, there is a unique global strong solution $(u, b)$ to the  PEM. Following  Definition \ref{7}, there is a global weak solution $(u_{\v},b_{\v})$  to the SMHD (\ref{2}).

Next, we perform the global Leray-Hopf weak solutions  of SMHD strong convergence to the global strong solutions of the PEM. As a preparation,  we need the  following proposition, which is fundamentally obtained by testing the SMHD against $(\U,u_3)$ and  $(\B, b_3)$.

\begin{prop}\label{d1}
Let $(\U_\v,u_{3,\v}, \B_\v,b_{3,\v})$ be the solution of SMHD, while the  $(\U,u_3,\B,b_3)$ be the solution of PEM,  with $(\U_{0},\B_{0})\in H^1(\Om)$ satisfying (\ref{s1}), (\ref{3}) and (\ref{30}),
 the integral equality holds
\begin{align}\label{p0}
&-\frac{\v^2}{2}\|u_3(t_0)\|_2^2-\frac{\v^2}{2}\|b_3(t_0)\|_2^2+\int_{Q_{t_0}}(-\U_{\v}\C\p_t\U-\B_\v\C\p_t\B
+\na\U_\v:\na\U+\v^2\na u_{3,\v}\C\na u_3\notag\\
&+\na\B_\v:\na\B+\v^2\na b_{3,\v}\C\na b_3)\,dxdydzdt+\Big(\int_{\Om}\U_{\v}\C\U+\B_{\v}\C\B+\v^2u_3u_{3,\v}+\v^2b_3b_{3,\v}\,dxdydz\Big)(t)\notag\\
=&\frac{\v^2}{2}\|u_{3,0}\|_2^2+\frac{\v^2}{2}\|b_{3,0}\|_2^2+\|\U_0\|_2^2+\|\B_0\|_2^2+\v^2\int_{Q_{t_0}}\Big(\int_0^z\p_t\U\,d\xi\Big)
\C\na_H U_{3,\v}\,dxdydzdt\notag\\
&+\v^2\int_{Q_{t_0}}\Big(\int_0^z\p_t\B\,d\xi\Big)\C\na_H B_{3,\v}\,dxdydzdt-\int_{Q_{t_0}}\Big[(u_{\v}\C\na)\U_{\v}\C\U-(b_{\v}\C\na)\B_{\v}\C\U\notag\\
&+\v^2u_{\v}\C\na u_{3,\v} u_3-\v^2b_{\v}\C\na b_{3,\v}u_3+u_{\v}\C\na\B_{\v}\C\B-b_{\v}\C\na\U_{\v}\C\B+\v^2\U_{\v}\C\na b_{3,\v} b_3\notag\\
&-\v^2b_{\v}\C\na u_{3,\v} b_3\Big]\,dxdydzdt,
\end{align}
for any $t_0\in [0,\infty)$, where $Q_{t_0}=\Om\times(0,t_0)$.
\end{prop}
\vskip .1in
\begin{proof}[Proof ] The definition of weak solutions to  the SMHD , the following integral identity holds
\begin{align*}
&\int_{Q_{t_0}}\Big[-(\U_{\v}\C\p_t\varphi_H+\v^2u_{3,\v}\p_t\varphi_3)-(\B_{\v}\C\p_t\psi_H+\v^2b_{3,\v}\p_t\psi_3)
+(u_{\v}\C\na)\U_{\v}\C\varphi_H\\
&-(b_{\v}\C\na)\B_{\v}\C\varphi_H+\v^2u_{\v}\C\na u_{3,\v}\varphi_3-\v^2b_{\v}\C\na b_{3,\v}\varphi_3+(u_{\v}\C\na)\B_{\v}\C\psi_H-(b_{\v}\C\na)\U_{\v}\C\psi_H\\
&+\v^2u_{\v}\C\na b_{3,\v}\psi_3-\v^2b_{\v}\C\na u_{3,\v}\psi_3+\na\U_{\v}:\na\varphi_H+\v^2\na u_{3,\v}\C\na\varphi_3+\na\B_{\v}:\na\psi_H\\
&+\v^2\na b_{3,\v}\C\na\psi_3\Big]\,dxdydzdt\\
&=\int_{\Om}(\U_0\C\varphi_H(\cdot,0)+\v^2u_{3,0}\varphi_3(\C,0))\,dxdydz+\int_{\Om}(\B_0\C\psi_H(\C,0)+\v^2b_{3,0}\psi_3(\C,0))\,dxdydz,
\end{align*}
for any periodic function $\varphi=(\varphi_H,\va_3)$ and $\ps=(\ps_H,\ps_3)$, with $\va_H=(\va_1,\va_2)$ and $\ps_H=(\ps_1,\ps_2)$, the divergence-free test functions   $\va$ and $\ps$ satisfy $\va\in C_0^{\infty}(\overline{\Om}\times[0,\infty))$ and $\ps\in C_0^\infty(\overline{\Om}\times[0,\infty))$, for any $t_0\in [0,\infty)$, where $Q_{t_0}=\Om\times(0,t_0)$.

Let $\chi(t)\in C_0^\infty([0,\infty))$, with $0\leq\chi(t)\leq 1$, and $\chi(0)=1$, and set $\va=(\U,u_3)\chi(t)$, $\ps=(\B,b_3)\chi(t)$, we note that with the density parameter, we can choose $\va$ and  $\ps$ as the testing function in the integral  identity above, with modifying the terms $\int_{Q_{t_0}} u_{3,\v}\p_t\va_3\,dxdydzdt$ and $\int_{Q_{t_0}} b_{3,\v}\p_t\ps_3\,dxdydzdt$ as
$$\int_{Q_{t_0}} u_{3,\v}\p_t(u_3\chi)\,dxdydzdt=\int_0^\infty\langle\p_t(u_3\chi),u_{3,\v}\rangle_{H^{-1}\times{H^1}}\,dt,$$
and
$$\int_{Q_{t_0}} b_{3,\v}\p_t(b_3\chi)\,dxdydzdt=\int_0^\infty\langle\p_t(b_3\chi),b_{3,\v}\rangle_{H^{-1}\times{H^1}}\,dt.$$


By taking $\va=(\U,u_3)\chi$, $\ps=(\B,b_3)\chi$ as a testing function, we get the following integral identity
\begin{align*}
&\int_{Q_{t_0}}\Big[\big(-\U_{\v}\C\p_t\U-\B_{\v}\C\p_t\B+\na\U_{\v}:\na\U+\v^2\na u_{3,\v}\C\na u_3+\na\B_{\v}:\na \B\\
&+\v^2\na b_{3,\v}\C\na b_3\big)\chi-\U_{\v}\C\U\chi'-\B_{\v}\C\B\chi'\Big]\,dxdydzdt
-\v^2\int_0^{\infty}\langle\p_t(u_3\chi),u_{3,\v}\rangle_{H^{-1}\times H^1}\,dt\\
&-\v^2\int_0^{\infty}\langle\p_t(b_3\chi),b_{3,\v}\rangle_{H^{-1}\times H^1}\,dt\\
=&-\int_{Q_{t_0}}\Big[(u_{\v}\C\na)\U_{\v}\C\U-(b_{\v}\C\na)\B_{\v}\C\U+\v^2 u_{\v}\C\na u_{3,\v} u_3-\v^2 b_{\v}\C\na b_{3,\v} u_3+u_{\v}\C\na\B_{\v}\C\B\\
&-b_{\v}\C\na\U_{\v}\C\B+\v^2u_{\v}\C\na b_{3,\v} b_3-\v^2b_{\v}\C\na u_{3,\v} b_3\Big]\chi\,dxdydzdt+\|\U_0\|_2^2+\v^2\|u_{3,0}\|_2^2\\
&+\|\B_0\|_2^2+\v^2\|b_{3,0}\|_2^2.
\end{align*}

 We rewrite the terms $\int_0^\infty\langle\p_t(u_3\chi),u_{3,\v}\rangle_{H^{-1}\times H^1}\,dt$ and $\int_0^\infty\langle\p_t(b_3\chi),b_{3,\v}\rangle_{H^{-1}\times H^1}\,dt$ as
$$\int_0^\infty\langle\p_t(u_3\chi),u_{3,\v}\rangle_{H^{-1}\times H^1}\,dt=\int_0^\infty\langle\p_t u_3,u_{3,\v}\rangle_{H^{-1}\times H^1}\chi\,dt+\int_{Q_{t_0}} u_3 u_{3,\v}\chi'\,dxdydzdt,$$
and
$$\int_0^\infty\langle\p_t(b_3\chi),b_{3,\v}\rangle_{H^{-1}\times H^1}\,dt=\int_0^\infty\langle\p_t b_3,b_{3,\v}\rangle_{H^{-1}\times H^1}\chi\,dt+\int_{Q_{t_0}} b_3 b_{3,\v}\chi'\,dxdydzdt.$$

Substituting in the previous identity gives that
\begin{align}\label{s4}
&\int_{Q_{t_0}}\big(-\U_{\v}\C\p_t\U-\B_{\v}\C\p_t\B+\na\U_{\v}:\na\U+{\v}^2\na u_{3,\v}\C\na u_3+\na\B_{\v}:\na\B\notag\\
&+{\v}^2\na b_{3,\v}\C\na b_3\big)\chi(t)\,dxdydzdt\notag\\
&-\v^2\int_0^\infty\langle\p_t u_3,u_{3,\v}\rangle_{H^{-1}\times H^1}\chi\,dt-\v^2\int_0^\infty\langle\p_t b_3,b_{3,\v}\rangle_{H^{-1}\times H^1}\chi\,dt\notag\\
&-\int_{Q_{t_0}}(\U_{\v}\C\U+\B_{\v}\C\B+\v^2u_3u_{3,\v}+\v^2b_3b_{3,\v})\chi'\,dxdydzdt\notag\\
=&-\int_{Q_{t_0}}\big[(u_{\v}\C\na)\U_{\v}\C\U-(b_{\v}\C\na)\B_{\v}\C\U+\v^2u_{\v}\C\na u_{3,\v}u_3-\v^2 b_{\v}\na b_{3,\v}u_3\notag\\
&+u_{\v}\C\na\B_{\v}\C\B-b_{\v}\C\na\U_{\v}\C\B+\v^2u_{\v}\C\na b_{3,\v} b_3-\v^2b_{\v}\C\na u_{3,\v} b_3\big]\chi\,dxdydzdt\notag\\
&+\|\U_0\|_2^2+\v^2\|u_{3,0}\|_2^2+\|\B_0\|_2^2+\v^2\|b_{3,0}\|_2^2,
\end{align}
for any $\chi\in C_0^\infty([0,\infty))$, with $0\leq\chi\leq1$ and $\chi(0)=1$.

 Choose $\chi_{\delta}\in C_0^\infty([0,t_0))$, such that $\chi_\delta\equiv1$ on $[0,t_0-\delta]$, $0\leq\chi_\delta\leq1$ on $[t_0-\delta,t_0)$ and $|\chi'_\delta|\leq\frac{2}{\delta}$ on $[0,t_0)$, where  $t_0\in(0,\infty)$, and $\delta\in(0,t_0)$ is a small positive number. As $\delta\rightarrow 0$, we have
\begin{align}\label{s5}
&\int_{Q_{t_0}}(\U_{\v}\C\U+\B_{\v}\C\B+\v^2u_3 u_{3,\v}+\v^2b_3 b_{3,\v})\chi'_{\delta}\,dxdydzdt\notag\\
&\rightarrow-\Big(\int_{\Om}(\U_{\v}\C\U+\B_{\v}\C\B+\v^2u_3 u_{3,\v}+\v^2b_3b_{3,\v})\,dxdydz\Big)(t_0),
\end{align}
\begin{align}\label{s6}
\int_0^\infty\langle\p_t u_3,u_{3,\v}\rangle_{H^{-1}\times H^{1}}\chi_{\delta}\,dt\rightarrow\int_0^{t_0}\langle\p_t u_3,u_{3,\v}\rangle_{H^{-1}\times H^1}\,dt,
\end{align}
and
\begin{align}\label{s7}
\int_0^\infty\langle\p_t b_3,b_{3,\v}\rangle_{H^{-1}\times H^{1}}\chi_{\delta}\,dt\rightarrow\int_0^{t_0}\langle\p_t b_3,b_{3,\v}\rangle_{H^{-1}\times H^1}\,dt.
\end{align}

 The effectiveness of (\ref{s6}) and (\ref{s7}) comes from the dominant convergence theorem for the integrals,
\begin{align*}
\langle\p_t u_3, u_{3,\v}\rangle=-\Big\langle\na_H\C\Big(\int_0^z\p_t\U\,d\xi\Big),u_{3,\v}\Big\rangle=\int_{\Om}\Big(\int_0^z\p_t\U\,d\xi\Big)\C\na_H u_{3,\v}\,dxdydz,
\end{align*}
and
\begin{align*}
\langle\p_t b_3, b_{3,\v}\rangle=-\Big\langle\na_H\C\Big(\int_0^z\p_t\B\,d\xi\Big),b_{3,\v}\Big\rangle=\int_{\Om}\Big(\int_0^z\p_t\B\,d\xi\Big)\C\na_H b_{3,\v}\,dxdydz,
\end{align*}
which implies $\langle\p_t u_3, u_{3,\v}\rangle\in L^1((0,t_0))$ and $\langle\p_t b_3, b_{3,\v}\rangle\in L^1((0,t_0))$. Here, for briefness, we have got rid of the subscript $H^{-1}\times H^1$. While for (\ref{s5}), we define
$$h(t):=\Big(\int_{\Om}(\U_{\v}\C\U+\B_{\v}\C\B+\v^2u_3u_{3,\v}+\v^2b_3b_{3,\v})\,dxdydz\Big)(t).$$
It is equivalent to show  $\int_{t_0-\delta}^{t_0} h(t)\chi'_\delta\,dt\rightarrow -h(t_0)$, as $ \delta\rightarrow 0$.

Recalling the regularities that $(u_{\v}, b_{\v})\in C_{w}([0, \infty);L^2(\Om))$  and  $(\U, \B)\in C([0,\infty);H^1(\Om))$, hence one has $(u_3,b_3)\in C([0, \infty);L^2(\Om))$, so $h$ is a continuous function on $[0,\infty)$. For any $\sigma>0$, there is a positive number $\rho$,   such that $|h(t)-h(t_0)|\leq\sigma$, and any $t\in[t_0-\rho,t_0]$. At present, for any $\delta\in (0,\rho)$, recalling that $\chi_{\delta}\equiv1$ on $[0,t_0-\delta]$, $\chi_\delta(t_0)=0$, and $|\chi'_{\delta}|\leq\frac{2}{\delta}$, on $[0,\infty)$, we have
\begin{align*}
\Bigg|\int_{t_0-\delta}^{t_0}h(t)\chi'_{\delta}(t)\,dt+h(t_0)\Bigg|&=\Bigg|\int_{t_0-\delta}^{t_0}(h(t)-h(t_0))\chi'_{\delta}(t)\,dt\Bigg|\\
&\leq\int_{t_0-\delta}^{t_0}|h(t)-h(t_0)||\chi'_{\delta}(t)|\,dt\\
&\leq2\sigma,
\end{align*}
which gives (\ref{s5}).

 Combining $u_3=-\int_0^z\na_H\C\U\,d\xi$ and recalling the regularities that  $u_3\in L_{loc}^2([0, \infty);H^1(\Om))$ and $\p_t u_3\in L_{loc}^2([0,\infty);H^{-1}(\Om))$, we get
\begin{align*}
\langle\p_t u_3, u_{3,\v}\rangle&=\langle\p_t u_3,u_{3,\v}-u_3\rangle+\langle\p_t u_3,u_3\rangle\\
&=\Bigg\langle-\na_H\C\Big(\int_0^z\p_t\U\,d\xi\Big), u_{3,\v}-u_3\Bigg\rangle+\langle\p_t u_3,u_3\rangle\\
&=\int_{\Om}\Big(\int_{0}^z\p_t\U\,d\xi\Big)\C\na_H U_{3,\v}\,dxdydz+\frac{1}{2}\frac{d}{dt}\|u_3\|_2^2,
\end{align*}
here we use the Lions-Magenes Lemma  \cite[see, e.g., pages 260-261]{RT}, we deduce
\begin{align*}
\int_0^{t_0}\langle\p_t u_3,u_{3,\v}\rangle\,dt=\int_{Q_{t_0}}\Big(\int_0^z\p_t\U\,d\xi\Big)\C\na_H U_{3,\v}\,dxdydzdt+\frac{1}{2}(\|u_3(t_0)\|_2^2-\|u_{3,0}\|_2^2).
\end{align*}
Similarly, we obtain
\begin{align*}
\int_0^{t_0}\langle\p_t b_3,b_{3,\v}\rangle\,dt=\int_{Q_{t_0}}\Big(\int_0^z\p_t\B\,d\xi\Big)\C\na_H B_{3,\v}\,dxdydzdt+\frac{1}{2}(\|b_3(t_0)\|_2^2-\|b_{3,0}\|_2^2).
\end{align*}
 Because of  the above equality and (\ref{s5})-(\ref{s7}), one can select $\chi=\chi_{\delta}$ in (\ref{s4}), as in the preceding paragraph, taking $\delta\rightarrow 0$ gives that
\begin{align}\label{s8}
&-\frac{\v^2}{2}\|u_3(t_0)\|_2^2-\frac{\v^2}{2}\|b_3(t_0)\|_2^2+\int_{Q_{t_0}}(-\U_{\v}\C\p_t\U-\B_{\v}\C\p_t\B
+\na\U_{\v}:\na\U+\v^2\na u_{3,\v}\C\na u_3\notag\\
&+\na\B_{\v}:\na\B+\v^2\na b_{3,\v}\C\na b_3)\,dxdydzdt+\Big(\int_{\Om}\U_{\v}\C\U+\B_{\v}\C\B+\v^2 u_3 u_{3,\v}\notag\\
&+\v^2b_3 b_{3,\v}\,dxdydz\Big)(t_0)\notag\\
=&\frac{\v^2}{2}\|u_{3,0}\|_2^2+\frac{\v^2}{2}\|b_{3,0}\|_2^2+\|\U_0\|_2^2+\|\B_0\|_2^2
+\v^2\int_{Q_{t_0}}\Big(\int_0^z\p_t\U\,d\xi\Big)\C\na_H U_{3,\v}\,dxdydzdt\notag\\
&+\v^2\int_{Q_{t_0}}\Big(\int_0^z\p_t\B\,d\xi\Big)\C\na_H B_{3,\v}\,dxdydzdt-\int_{Q_{t_0}}\Big[(u_{\v}\C\na)\U_{\v}\C\U-(b_{\v}\C\na)\B_{\v}\C\U\notag\\
&+\v^2u_{\v}\C\na u_{3,\v}u_3-\v^2b_{\v}\C\na b_{3,\v}u_3+u_{\v}\C\na\B_{\v}\C\B-b_{\v}\C\na\U_{\v}\C\B+\v^2\U_{\v}\C\na b_{3,\v} b_3\notag\\
&-\v^2b_{\v}\C\na u_{3,\v} b_3\Big]\,dxdydzdt,
\end{align}
for any $t_0\in[0,\infty)$. This completes the proof.
\end{proof}

Now, we can estimate the difference between $(\U_{\v}, u_{3,\v}, \B_{\v},b_{3,\v})$ and  $(\U,u_3, \B, b_3)$.

\begin{prop}\label{d2}
Under the same conditions as in Proposition \ref{d1}, we denote $(\du_{\v},U_{3,\v}):=(\U_{\v}-\U, u_{3,\v}-u_3)$ and  $(\db_{\v},B_{3,\v}):=(\B_{\v}-\B, b_{3,\v}-b_3)$, the following estimate holds
\begin{align*}
&\sup_{0\leq t<\infty}\big(\|\du_{\v}\|_2^2+\|\db_{\v}\|_2^2+\v^2\|U_{3,\v}\|_2^2+\v^2\|B_{3,\v}\|_2^2\big)(t)
+\int_0^{\infty}\big(\|\na\du_{\v}\|_2^2+\|\na\db_{\v}\|_2^2\\
&+\v^2\|\na U_{3,\v}\|_2^2+\v^2\|\na B_{3,\v}\|_2^2\big)\,ds\\
\leq &\v^2C(\|\U_0\|_2^2+\|\B_0\|_2^2+\v^2\|u_{3,0}\|_2^2+\v^2\|b_{3,0}\|_2^2+1)^2,
\end{align*}
where $C$ is a positive constants depending only on $\|\U_0\|_{H^1},\|\B_0\|_{H^1},L_1$ and $L_2$.
\end{prop}
\begin{proof}[Proof] Multiplying  $(\ref{b5})_1$ and $(\ref{b5})_3$  by  $\U_{\v}$ and $\B_{\v}$, respectively,  and taking the $L^2$ norm on $Q_{t_0}$, if follows from integration by parts that
\begin{align}\label{p1}
&\int_{Q_{t_0}}\big(\p_t\U\C\U_{\v}+\p_t\B\C\B_{\v}+\na\U:\na\U_{\v}+\na\B:\na\B_{\v}\big)\,dxdydzdt\notag\\
=&-\int_{Q_{t_0}}u\C\na\U\C\U_{\v}\,dxdydzdt+\int_{Q_{t_0}}b\C\na\B\C\U_{\v}\,dxdydzdt\notag\\&-\int_{Q_{t_0}}u\C\na\B\C\B_{\v}\,dxdydzdt
+\int_{Q_{t_0}}b\C\na\U\C\B_{\v}\,dxdydzdt,
\end{align}
for any $t_0\in [0,\infty)$, applying the first equation and the third equation of  (\ref{b5}) by $\U$ and $\B$, respectively, integrating the resultant over $Q_{t_0}$, it follows from integration by parts that
\begin{align}\label{p2}
&\frac{1}{2}\|\U(t_0)\|_2^2+\frac{1}{2}\|\B(t_0)\|_2^2+\int_0^{t_0}\|\na\U\|_2^2\,dt+\int_0^{t_0}\|\na\B\|_2^2\,dt\notag\\
=&\frac{1}{2}\|\U_0\|_2^2+\frac{1}{2}\|\B_0\|_2^2,
\end{align}
for any $t_0\in[0,\infty)$. Recalling the definition of Leray-Hopf weak solutions to the SMHD, we have
\begin{align}\label{p3}
&\frac{1}{2}\big(\|\U_{\v}(t_0)\|_2^2+\|\B(t_0)\|_2^2+\v^2\|u_{3,\v}(t_0)\|_2^2+\v^2\|b_{3,\v}(t_0)\|_2^2\big)\notag\\
&+\int_0^{t_0}(\|\na\U_{\v}\|_2^2+\|\na\B_{\v}\|_2^2+\v^2\|\na u_{3,\v}\|_2^2+\v^2\|\na b_{3,\v}\|_2^2)\,ds\notag\\
\leq &\frac{1}{2}\big(\|\U_0\|_2^2+\|\B_0\|_2^2+\v^2\|u_{3,0}\|_2^2+\v^2\|b_{3,0}\|_2^2\big),
\end{align}
for a.e. $t_0\in[0,\infty)$.

Adding (\ref{p2}) and (\ref{p3}), then subtracting from the  (\ref{p0}) (choose $t=t_0$ there) and
 (\ref{p1}), we get
\begin{align}\label{p4}
&\frac{1}{2}\Big(\|\du_{\v}\|_2^2+\|\db_{\v}\|_2^2+\v^2\|U_{3,\v}\|_2^2+\v^2\|B_{3,\v}\|_2^2\Big)(t_0)\notag\\
&+\int_0^{t_0}\Big(\|\na\du_{\v}\|_2^2+\|\na\db_{\v}\|_2^2+\v^2\|\na U_{3,\v}\|_2^2+\v^2\|\na B_{3,\v}\|_2^2\Big)\,dt\notag\\
\leq&-\v^2\int_{Q_{t_0}}\Big[\Big(\int_0^z\p_t\U\,d\xi\Big)\C\na_H U_{3,\v}+\na u_3\C\na U_{3,\v}\Big]\,dxdydzdt\notag\\
&-\v^2\int_{Q_{t_0}}\Big[\Big(\int_0^z\p_t\B\,d\xi\Big)\C\na_H B_{3,\v}+\na b_3\C\na B_{3,\v}\Big]\,dxdydzdt\notag\\
&+\int_{Q_{t_0}}\big[(u_{\v}\C\na)\U_{\v}\C\U+(u\C\na)\U\C\U_{\v}\big]\,dxdydzdt
-\int_{Q_{t_0}}\big[(b_{\v}\C\na)\B_{\v}\C\U+(b\C\na)\B\C\U_{\v}\big]\,dxdydzdt\notag\\
&-\int_{Q_{t_0}}\big[(b_{\v}\C\na)\U_{\v}\C\B+(b\C\na)\U\C\B_{\v}\big]\,dxdydzdt
+\int_{Q_{t_0}}\big[(u_{\v}\C\na)\B_{\v}\C\B+(u\C\na)\B\C\B_{\v}\big]\,dxdydzdt
\notag\\
&+\v^2\int_{Q_{t_0}}u_{\v}\C\na u_{3,\v} u_3\,dxdydzdt-\v^2\int_{Q_{t_0}}b_{\v}\C\na b_{3,\v} u_3\,dxdydzdt\notag\\
&-\v^2\int_{Q_{t_0}}b_{\v}\C\na u_{3,\v} b_3\,dxdydzdt+\v^2\int_{Q_{t_0}}u_{\v}\C\na b_{3,\v} b_3\,dxdydzdt\notag\\
:=&I_1+I_2+I_3+I_4+I_5+I_6+I_7+I_8+I_9+I_{10},
\end{align}
for a.e. $t_0\in [0,\infty)$. Using the H\"{o}lder,  Cauchy-Schwarz inequalities and Proposition \ref{5}, we infer that
 \begin{align*}
I_1\leq&\v^2(\|\p_t\U\|_{L^2(Q_{t_0})}+\|\na u_3\|_{L^2(Q_{t_0})})\|\na U_{3,\v}\|_{L^2(Q_{t_0})}\\
\leq&\frac{\v^2}{6}\|\na U_{3,\v}\|^2_{L^2(Q_{t_0})}+C(\|\U_0\|_{H^1},L_1,L_2)\v^2,
\end{align*}
similarly, we have
\begin{align*}
I_2\leq&\v^2(\|\p_t\B\|_{L^2(Q_{t_0})}+\|\na b_3\|_{L^2(Q_{t_0})})\|\na B_{3,\v}\|_{L^2(Q_{t_0})}\\
\leq &\frac{\v^2}{6}\|\na B_{3,\v}\|^2_{L^2(Q_{t_0})}+C(\|\B_0\|_{H^1},L_1,L_2)\v^2.
\end{align*}

Next, we are going to estimate $I_3$. Due to the incompressibility conditions, it follows from integration by parts that
\begin{align*}
I_3=&\int_{Q_{t_0}}[(u_{\v}\C\na)\U_{\v}\C\U+(u\C\na)\U\C\U_{\v}]\,dxdydzdt\\
=&\int_{Q_{t_0}}[(u_{\v}\C\na)\U_{\v}\C\U-(u\C\na)\U_{\v}\C\U]\,dxdydzdt\\
=&\int_{Q_{t_0}}[(u_{\v}-u)\C\na]\U_{\v}\C\U\,dxdydzdt\\
=&\int_{Q_{t_0}}[(u_{\v}-u)\C\na]\du_{\v}\C\U\,dxdydzdt,\\
=&\int_{Q_{t_0}}[(\du_{\v}\C\na_H)\du_{\v}\C\U]\,dxdydzdt+\int_{Q_{t_0}}U_{3,\v}\p_z\du_{\v}\C\U\,dxdydzdt,\\
:=&I_{31}+I_{32}.
\end{align*}

Taking advantage of  the H\"{o}lder, Sobolev and Young inequalities, we have
\begin{align*}
I_{31}&=\int_{Q_{t_0}}[(\du_{\v}\C\na_H)\du_{\v}\C\U]\,dxdydzdt\\
&\leq C\int_0^{t_0}\|\du_{\v}\|_3\|\na\du_{\v}\|_2\|\U\|_6\,dt\\
&\leq C\int_0^{t_0}\|\du_{\v}\|_2^{\frac{1}{2}}\|\na\du_{\v}\|_2^{\frac{3}{2}}\|\na\U\|_2\,dt\\
&\leq \frac{1}{16}\|\na\du_{\v}\|_{L^2(Q_{t_0})}^2+C\int_0^{t_0}\|\na\U\|_2^4\|\du_{\v}\|_2^2\,dt.
\end{align*}

Integration by parts yields
\begin{align*}
I_{32}=&\int_{Q_{t_0}}U_{3,\v}\p_z\du_{\v}\C\U\,dxdydzdt\\
=&-\int_{Q_{t_0}}[\p_z U_{3,\v}\du_{\v}\C\U+U_{3,\v}\du_{\v}\C\p_z\U]\,dxdydzdt\\
=&\int_{Q_{t_0}}[(\na_H\C\du_{\v})\du_{\v}\C\U-U_{3,\v}\du_{\v}\C\p_z\U]\,dxdydzdt\\
:=&I_{321}+I_{322}.
\end{align*}

The same arguments as for $I_{31}$ yield\begin{align*}
I_{321}=&\int_{Q_{t_0}}(\na_H\C\du_{\v})(\du_{\v}\C\U)\,dxdydzdt\\
\leq&\frac{1}{16}\|\na\du_{\v}\|_{L^2(Q_{t_0})}^2+C\int_0^{t_0}\|\na\U\|_2^4\|\du_{\v}\|_2^2\,dt.
\end{align*}

Applying Lemma \ref{b7},  the Poincar\'e and Young inequalities, we arrive that
\begin{align*}
I_{322}=&-\int_{Q_{t_0}}U_{3,\v}\du_{\v}\C\p_z\U\,dxdydzdt\\
=&\int_0^{t_0}\int_{\Om}\Big(\int_0^z\na_H\C\du_{\v}\,d\xi\Big)(\du_{\v}\C\p_z\U)\,dxdydzdt\\
\leq&\int_0^{t_0}\int_M\Big(\int_{-1}^1|\na_H\du_{\v}|\,dz\Big)\Big(\int_{-1}^1|\du_{\v}||\p_z\U|\,dz\Big)\,dxdydt\\
\leq&C\int_0^{t_0}\|\na\du_{\v}\|_2^{\frac{3}{2}}\|\du_{\v}\|_2^{\frac{1}{2}}\|\na\U\|_2^{\frac{1}{2}}
\|\De\U\|_2^{\frac{1}{2}}\,dt\\
\leq&\frac{1}{16}\|\na\du_{\v}\|_{L^2(Q_{t_0})}^2+C\int_0^{t_0}\|\na\U\|_2^2\|\De\U\|_2^2\|\du_{\v}\|_2^2\,dt.
\end{align*}

Thanks to the estimates for $I_{31}$, $I_{321}$ and $I_{322}$, we can bound $I_3$ as
\begin{align*}
I_3\leq\frac{3}{16}\|\na\du_{\v}\|_{L^2(Q_{t_0})}^2+C\int_0^{t_0}\|\na\U\|_2^2\|\De\U\|_2^2\|\du_{\v}\|_2^2\,dt.
\end{align*}

For the sake of simplicity, we sum up the following two terms, then
\begin{align*}
I_4+I_5=&-\int_{Q_{t_0}}[b_{\v}\C\na\B_{\v}\C\U+b\C\na\B\C\U_{\v}]\,dxdydzdt
-\int_{Q_{t_0}}[b_{\v}\C\na\U_{\v}\C\B+b\C\na\U\C\B_{\v}]\,dxdydzdt\\
=&-\int_{Q_{t_0}}[b_{\v}\C\na\B_{\v}\C\U+b\C\na\U\C\B_{\v}]\,dxdydzdt
-\int_{Q_{t_0}}[b_{\v}\C\na\U_{\v}\C\B+b\C\na\B\C\U_{\v}]\,dxdydzdt\\
=&-\int_{Q_{t_0}}[b_{\v}\C\na\B_{\v}\C\U-b\C\na \B_{\v}\C\U]\,dxdydzdt
-\int_{Q_{t_0}}[b_{\v}\C\na\U_{\v}\C\B-b\C\na\U_{\v}\C\B]\,dxdydzdt\\
:=&A+B.
\end{align*}

To bound $A$, we decompose it into two pieces

\begin{align*}
A=&-\int_{Q_{t_0}}[b_{\v}\C\na\B_{\v}\C\U-b\C\na \B_{\v}\C\U]\,dxdydzdt\\
=&-\int_{Q_{t_0}}[(b_{\v}-b)\C\na\B_{\v}\C\U]\,dxdydzdt\\
=&-\int_{Q_{t_0}}B_{\v}\C\na\B_{\v}\C\U\,dxdydzdt\\
=&-\int_{Q_{t_0}}\db_{\v}\C\na_{H}\B_{\v}\C\U\,dxdydzdt-\int_{Q_{t_0}}B_{3,\v}\p_z\B_{\v}\C\U\,dxdydzdt\\
:=&A_1+A_2.
\end{align*}

So that the first part of $A$ can be estimated
\begin{align*}
A_1=&-\int_{Q_{t_0}}\db_{\v}\C\na_H\B_{\v}\C\U\,dxdydzdt\\
\leq&\int_0^{t_0}\|\db_{\v}\|_6\|\na_H\B_{\v}\|_2\|\U\|_3\,dt\\
\leq&\int_0^{t_0}\|\na\db_{\v}\|_2\|\na\B_{\v}\|_2\|\U\|_2^{\frac{1}{2}}\|\na\U\|_2^{\frac{1}{2}}\,dt\\
\leq&\frac{1}{22}\|\na\db_{\v}\|_{L^2(Q_{t_0})}^2+C\int_0^{t_0}\|\na\B_{\v}\|_2^2\|\U\|_2\|\na\U\|_2\,dt\\
\leq&\frac{1}{22}\|\na\db_{\v}\|_{L^2(Q_{t_0})}^2+C\int_0^{t_0}\|\na\B_{\v}\|_2^2(\|\U\|_2^2+\|\na\U\|_2^2)\,dt\\
\leq&\frac{1}{22}\|\na\db_{\v}\|_{L^2(Q_{t_0})}^2
+C\int_0^{t_0}\|\na\B_{\v}\|_2^2\|\U\|_2^2\,dt+C\int_0^{t_0}\|\na\B_{\v}\|_2^2\|\na\U\|_2^2\,dt.
\end{align*}

Applying integration by parts gives that
\begin{align*}
A_2=&-\int_{Q_{t_0}}B_{3,\v}(\p_z\B_{\v}\C\U)\,dxdydzdt\\
=&\int_{Q_{t_0}}\p_z B_{3,\v}(\B_{\v}\C\U)\,dxdydzdt+\int_{Q_{t_0}}B_{3,\v}\B_{\v}\C\p_z\U\,dxdydzdt\\
=&-\int_{Q_{t_0}}(\na_{H}\C\db_{\v})\B_{\v}\C\U\,dxdydzdt+\int_{Q_{t_0}}B_{3,\v}\B_{\v}\C\p_z\U\,dxdydzdt\\
:=&A_{21}+A_{22}.
\end{align*}

The first part $A_{21}$ can be estimated as follows
\begin{align*}
A_{21}=&-\int_{Q_{t_0}}(\na_H\C\db_{\v})\B_{\v}\C\U\,dxdydzdt\\
\leq&\int_0^{t_0}\|\na_H\C\db_{\v}\|_2\|\B_{\v}\|_6\|\U\|_3\,dt\\
\leq&C\int_0^{t_0}\|\na\db_{\v}\|_2\|\na\B_{\v}\|_2\|\U\|_2^{\frac{1}{2}}\|\na\U\|_2^{\frac{1}{2}}\,dt\\
\leq&\frac{1}{22}\|\na\db_{\v}\|_{L^2(Q_{t_0})}^2+C\int_0^{t_0}\|\na\B_{\v}\|_2^2\|\U\|_2\|\na\U\|_2\,dt\\
\leq&\frac{1}{22}\|\na\db_{\v}\|_{L^2(Q_{t_0})}^2+C\int_0^{t_0}\|\na\B_{\v}\|_2^2(\|\U\|_2^2+\|\na\U\|_2^2)\,dt\\
\leq&\frac{1}{22}\|\na\db_{\v}\|_{L^2(Q_{t_0})}^2+C\int_0^{t_0}\|\na\B_{\v}\|_2^2\|\U\|_2^2\,dt+C\int_0^{t_0}\|
\na\B_{\v}\|_2^2\|\na\U\|_2^2\,dt.
\end{align*}

For the second part $A_{22}$, we use Lemma \ref{b7}, the Poincar\'e and Young inequalities and  obtain
\begin{align*}
A_{22}=&\int_{Q_{t_0}}\db_{3,\v}\B_{\v}\C\p_z\U\,dxdydzdt\\
=&\int_0^{t_0}\int_{\Om}\Big(\int_0^z\na_H\C\db_{\v}\,d\xi\Big)\B_{\v}\C\p_z\U\,dxdydzdt\\
\leq&\int_0^{t_0}\int_{M}\Big(\int_{-1}^1|\na_H\db_{\v}|\,dz\Big)\Big(\int_{-1}^1|\B_{\v}||\p_z\U|\,dz\Big)\,dxdydt\\
\leq&\int_0^{t_0}\|\na_H\db_{\v}\|_2\|\B_{\v}\|_2^{\frac{1}{2}}\|\na\B_{\v}\|_2^{\f}\|\na\U\|_2^{\f}\|\De\U\|_2^{\f}\,dt\\
\leq&\frac{1}{22}\|\na\db_{\v}\|_{L^2(Q_{t_0})}^2+C\int_0^{t_0}\|\B_{\v}\|_2\|\na\B_{\v}\|_2\|\na\U\|_2\|\De\U\|_2\,dt\\
\leq&\frac{1}{22}\|\na\db_{\v}\|_{L^2(Q_{t_0})}^2+C\int_0^{t_0}(\|\B_{\v}\|_2^2\|\na\B_{\v}\|_2^2+\|\na\U\|_2^2\|\De\U\|_2^2)\,dt.
\end{align*}

To deal with $B$, we break it down, one has
\begin{align*}
B=&-\int_{Q_{t_0}}[b_{\v}\C\na\U_{\v}\C\B-b\C\na\U_{\v}\C\B]\,dxdydzdt\\
=&-\int_{Q_{t_0}}[(b_{\v}-b)\C\na\U_{\v}\C\B]\,dxdydzdt\\
=&-\int_{Q_{t_0}}B_{\v}\C\na\U_{\v}\C\B\,dxdydzdt\\
=&-\int_{Q_{t_0}}\db_{\v}\C\na_H\U_{\v}\C\B\,dxdydzdt-\int_{Q_{t_0}}\db_{3,\v}\p_z\U_{\v}\C\B\,dxdydzdt\\
:=&B_1+B_2.
\end{align*}

Along the similar argument for the estimate of $A$,
\begin{align*}
B_1=&-\int_{Q_{t_0}}\db_{\v}\C\na_H\U_{\v}\C\B\,dxdydzdt\\
\leq&\int_0^{t_0}\|\db_{\v}\|_6\|\na_H\U_{\v}\|_2\|\B\|_3\,dt\\
\leq&\int_0^{t_0}\|\na\db_{\v}\|_2\|\na\U_{\v}\|_2\|\B\|_2^{\frac{1}{2}}\|\na\B\|_2^{\frac{1}{2}}\,dt\\
\leq&\frac{1}{22}\|\na\db_{\v}\|_{L^2(Q_{t_0})}^2+C\int_0^{t_0}\|\na\U_{\v}\|_2^2\|\B\|_2\|\na\B\|_2\,dt\\
\leq&\frac{1}{22}\|\na\db_{\v}\|_{L^2(Q_{t_0})}^2+C\int_0^{t_0}\|\na\U_{\v}\|_2^2(\|\B\|_2^2+\|\na\B\|_2^2)\,dt\\
\leq&\frac{1}{22}\|\na\db_{\v}\|_{L^2(Q_{t_0})}^2
+C\int_0^{t_0}\|\na\U_{\v}\|_2^2\|\B\|_2^2\,dt+C\int_0^{t_0}\|\na\U_{\v}\|_2^2\|\na\B\|_2^2\,dt.
\end{align*}

It follows from integration by parts that
\begin{align*}
B_2=&-\int_{Q_{t_0}}\db_{3,\v}\p_z\U_{\v}\C\B\,dxdydzdt\\
=&\int_{Q_{t_0}}\p_z\db_{3,\v}\U_{\v}\C\B\,dxdydzdt+\int_{Q_{t_0}}\db_{3,\v}\U_{\v}\p_z\C\B\,dxdydzdt\\
=&-\int_{Q_{t_0}}(\na_{H}\C\db_{\v})\U_{\v}\C\B\,dxdydzdt+\int_{Q_{t_0}}\db_{3,\v}\U_{\v}\C\p_z\B\,dxdydzdt\\
:=&B_{21}+B_{22}.
\end{align*}

The estimate for $B_{21}$ is given as follows. By the  H\"{o}lder and Young inequalities, we deduce
\begin{align*}
B_{21}=&-\int_{Q_{t_0}}(\na_H\C\db_{\v})\U_{\v}
\C\B\,dxdydzdt\\
\leq&\int_0^{t_0}\|\na_H\C\db_{\v}\|_2\|\U_{\v}\|_6\|\B\|_3\,dt\\
\leq&C\int_0^{t_0}\|\na\db_{\v}\|_2\|\na\U_{\v}\|_2\|\B\|_2^{\frac{1}{2}}\|\na\B\|_2^{\frac{1}{2}}\,dt\\
\leq&\frac{1}{22}\|\na\db_{\v}\|_{L^2(Q_{t_0})}^2+C\int_0^{t_0}\|\na\U_{\v}\|_2^2\|\B\|_2\|\na\B\|_2\,dt\\
\leq&\frac{1}{22}\|\na\db_{\v}\|_{L^2(Q_{t_0})}^2+C\int_0^{t_0}\|\na\U_{\v}\|_2^2(\|\B\|_2^2+\|\na\B\|_2^2)\,dt\\
\leq&\frac{1}{22}\|\na\db_{\v}\|_{L^2(Q_{t_0})}^2+C\int_0^{t_0}\|\na\U_{\v}\|_2^2\|\B\|_2^2\,dt+C\int_0^{t_0}\|
\na\U_{\v}\|_2^2\|\na\B\|_2^2\,dt.
\end{align*}

For the second term of $B_2$, taking advantage of  Lemma \ref{b7}, the Poincar\'e and Young inequalities that
\begin{align*}
B_{22}=&\int_{Q_{t_0}}\db_{3,\v}\U_{\v}\C\p_z\B\,dxdydzdt\\
=&\int_0^{t_0}\int_{\Om}\Big(\int_0^z\na_H\C\db_{\v}\,d\xi\Big)\U_{\v}\p_z\B\,dxdydzdt\\
=&\int_0^{t_0}\int_{M}\Big(\int_{-1}^1|\na_H\db_{\v}|\,dz\Big)\Big(\int_{-1}^1|\U_{\v}|\p_z\B|\,dz\Big)\,dxdydt\\
\leq&\int_0^{t_0}\|\na_H\db_{\v}\|_2\|\U_{\v}\|_2^{\frac{1}{2}}\|\na\U_{\v}\|_2^{\f}\|\na\B\|_2^{\f}\|\De\B\|_2^{\f}\,dt\\
\leq&\frac{1}{22}\|\na\db_{\v}\|_{L^2(Q_{t_0})}^2+C\int_0^{t_0}\|\U_{\v}\|_2\|\na\U_{\v}\|_2\|\na\B\|_2\|\De\B\|_2\,dt\\
\leq&\frac{1}{22}\|\na\db_{\v}\|_{L^2(Q_{t_0})}^2+C\int_0^{t_0}(\|\U_{\v}\|_2^2\|\na\U_{\v}\|_2^2+\|\na\B\|_2^2\|\De\B\|_2^2)\,dt.
\end{align*}

To deal with $I_6$, we break it down
\begin{align*}
I_6=&\int_{Q_{t_0}}[(u_{\v}\C\na)\B_{\v}\C\B+(u\C\na)\B\C\B_{\v}]\,dxdydzdt\\
=&\int_{Q_{t_0}}[(u_{\v}\C\na)\B_{\v}\C\B-(u\C\na)\B_{\v}\C\B]\,dxdydzdt\\
=&\int_{Q_{t_0}}[(u_{\v}-u)\C\na]\B_{\v}\C\B\,dxdydzdt\\
=&\int_{Q_{t_0}}[U_{\v}\C\na]\db_{\v}\C\B\,dxdydzdt\\
=&\int_{Q_{t_0}}(\du_{\v}\C\na_H)\db_{\v}\C\B\,dxdydzdt+\int_{Q_{t_0}}U_{3,\v}\p_z\db_{\v}\C\B\,dxdydzdt.\\
:=&I_{61}+I_{62}.
\end{align*}

It follows from the Sobolev and Young inequalities that
\begin{align*}
I_{61}&=\int_{Q_{t_0}}(\du_{\v}\C\na_H)\db_{\v}\C\B\,dxdydzdt\\
&\leq\int_0^{t_0}\|\du_{\v}\|_3\|\na_H\db_{\v}\|_2\|\B\|_6\,dt\\
&\leq\int_0^{t_0}\|\du_{\v}\|_2^{\frac{1}{2}}\|\na\du_{\v}\|_2^{\frac{1}{2}}\|\na_H\db_{\v}\|_2\|\na\B\|_2\,dt\\
&\leq\frac{1}{16}\|\na\du_{\v}\|_2^2+\frac{1}{22}\|\na\db_{\v}\|_2^2+C\int_0^{t_0}\|\na\B\|_2^4\|\du_{\v}\|_2^2\,dt,
\end{align*}
for $I_{62}$, we further decompose it into two pieces
\begin{align*}
I_{62}=&\int_{Q_{t_0}}U_{3,\v}\p_z\db_{\v}\C\B\,dxdydzdt\\
=&-\int_{Q_{t_0}}[\p_z U_{3,\v}\db_{\v}\C\B+U_{3,\v}\db_{\v}\C\p_z\B]\,dxdydzdt\\
=&\int_{Q_{t_0}}[\na_H\C\du_{\v}\db_{\v}\C\B-U_{3,\v}\db_{\v}\C\p_z\B]\,dxdydzdt\\
:=&I_{621}+I_{622}.
\end{align*}

By the H\"{o}lder and Young inequalities, we deduce
\begin{align*}
I_{621}=&\int_{Q_{t_0}}\na_H\C\du_{\v}\db_{\v}\C\B\,dxdydzdt\\
\leq&\int_0^{t_0}\|\na_H\du_{\v}\|_2\|\db_{\v}\|_3\|\B\|_6\,dt\\
\leq&C\int_0^{t_0}\|\na_H\du_{\v}\|_2\|\db_{\v}\|_2^{\frac{1}{2}}\|\na\db_{\v}\|_2^{\frac{1}{2}}\|\na\B\|_2\,dt\\
\leq&\frac{1}{16}\|\na\du_{\v}\|_2^2+\frac{1}{22}\|\na\db_{\v}\|_2^2+C\int_0^{t_0}\|\na\B\|_2^4\|\db_{\v}\|_2^2\,dt.
\end{align*}

 A Similar argument to that for $B_{22}$, we have
\begin{align*}
I_{622}=&-\int_{Q_{t_0}}U_{3,\v}\db_{\v}\C\p_z\B\,dxdydzdt\\
&=\int_0^{t_0}\int_{\Om}\Big(\int_0^z\na_H\C\du_{\v}\,d\xi\Big)(\db_{\v}\C\p_z\B)\,dxdydzdt\\
&\leq\int_0^{t_0}\int_{M}\Big(\int_{-1}^1|\na_H\du_{\v}|\,dz\Big)\Big(\int_{-1}^{1}|\db_{\v}||\p_z\B|\,dz\Big)\,dxdydt\\
&\leq C\int_0^{t_0}\|\na_H\du_{\v}\|_2\|\db_{\v}\|_2^{\frac{1}{2}}\|\na\db_{\v}\|_2^{\frac{1}{2}}
\|\na\B\|_2^{\frac{1}{2}}\|\De\B\|_2^{\frac{1}{2}}\,dt\\
&\leq\frac{1}{16}\|\na\du_{\v}\|_2^2+\frac{1}{22}\|\na\db_{\v}\|_2^2+C\int_0^{t_0}\|\na\B\|_2^2\|\De\B\|_2^2\|\db_{\v}\|_2^2\,dt.
\end{align*}

Using the incompressibility conditions, we deduce that
\begin{align*}
I_7=&\v^2\int_{Q_{t_0}}u_{\v}\C\na u_{3,\v}u_3\,dxdydzdt\\
=&\v^2\int_{Q_{t_0}}u_{\v}\C\na U_{3,\v} u_3\,dxdydzdt\\
=&\v^2\int_{Q_{t_0}}[\U_{\v}\C\na_HU_{3,\v}-u_{3,\v}\na_H\C\du_{\v}]u_3\,dxdydzdt\\
\leq&\v^2\int_0^{t_0}\int_M\Big(\int_{-1}^1(|\U_{\v}||\na_{H} U_{3,\v}|+|u_{3,\v}||\na_{H}\du_{\v}|\,dz\Big)\Big(\int_{-1}^1|\na_{H}\U|\,dz\Big)\,dxdydt,
\end{align*}
applying  Lemma \ref{b7}, the Poincar\'e and Young inequalities that
\begin{align*}
I_7=&C\v^2\int_0^{t_0}\big(\|\U_{\v}\|_2^{\f}\|\na\U_{\v}\|_2^{\f}\|\na U_{3,\v}\|_2
+\|u_{3,\v}\|_2^{\f}\|\na u_{3,\v}\|_2^{\f}\|\na_H \du_{\v}\|_2\big)\|\na\U\|_2^{\f}\|\De\U\|_2^{\f}\,dt\\
\leq&\frac{1}{16}\|\na\du_{\v}\|_{L^2(Q_{t_0})}^2+\frac{1}{6}\v^2\|\na U_{3,\v}\|_{L^2(Q_{t_0})}^2\\
&+C\v^2\int_0^{t_0}(\|\U_{\v}\|_2^2\|\na\U_{\v}\|_2^2+\|\na\U\|_2^2\|\De\U\|_2^2+\v^2\|u_{3,\v}\|_2^2\|\na u_{3,\v}\|_2^2)\,dt,
\end{align*}
from which, recalling (\ref{p3}) and by Proposition \ref{5}, we have
\begin{align*}
I_7\leq&\frac{1}{16}\|\na \du_{\v}\|_{L^2(Q_{t_0})}^2+\frac{1}{6}\v^2\|\na U_{3,\v}\|_{L^2(Q_{t_0})}^2+C\v^2\Big[(\|\U_0\|_2^2+\v^2\|u_{3,0}\|_2^2)^2
+C(\|\U_0\|_{H^1}, L_1, L_2)\Big].
\end{align*}

For $I_8+I_9$, applying the divergence free conditions gives that
\begin{align*}
I_8+I_9=&-\v^2\int_{Q_{t_0}}b_{\v}\C\na b_{3,\v}u_3\,dxdydzdt+\v^2\int_{Q_{t_0}}b_{\v}\C\na b_3 u_3\,dxdydzdt\\
&-\v^2\int_{Q_{t_0}}b_{\v}\C\na b_3 u_3\,dxdydzdt-\v^2\int_{Q_{t_0}}b_{\v}\C\na u_{3,\v}b_3\,dxdydzdt\\
&\v^2\int_{Q_{t_0}}b_{\v}\C\na u_3 b_3\,dxdydzdt-\v^2\int_{Q_{t_0}}b_{\v}\C\na u_3 b_3\,dxdydzdt\\
=&-\v^2\int_{Q_{t_0}}b_{\v}\C\na b_{3,\v}u_3\,dxdydzdt+\v^2\int_{Q_{t_0}}b_{\v}\C\na b_3 u_3\,dxdydzdt\\
&-\v^2\int_{Q_{t_0}}b_{\v}\C\na u_{3,\v}b_3\,dxdydzdt+\v^2\int_{Q_{t_0}}b_{\v}\C\na u_3 b_3\,dxdydzdt\\
:=&I_{81}+I_{82}+I_{91}+I_{92}.
\end{align*}

Employing Lemma \ref{b7}, the H\"{o}lder and Young inequalities we infer that
\begin{align*}
I_{81}+I_{82}=&-\v^2\int_{Q_{t_0}}(b_{\v}\C\na b_{3,\v} u_3-b_{\v}\C\na b_3 u_3)\,dxdydzdt\\
=&-\v^2\int_{Q_{t_0}}(b_{\v}\C\na B_{3,\v})u_3\,dxdydzdt\\
=&-\v^2\int_{Q_{t_0}}[\B_{\v}\C\na_H B_{3,\v}-b_{3,\v}\na_H\C\db_{\v}]u_3\,dxdydzdt\\
\leq&\v^2\int_0^{t_0}\int_{M}\Big(\int_{-1}^1(|\B_{\v}||\na_H B_{3,\v}|+|b_{3,\v}||\na_H\db_{\v}|)\,dz\Big)\Big(\int_{-1}^1|\na_H\U|\,dz\Big)\,dxdydt,\\
\leq&C\v^2\int_0^{t_0}(\|\B_{\v}\|_2^{\f}\|\na\B_{\v}\|_2^{\f}\|\na B_{3,\v}\|_2+\|b_{3,\v}\|_2^{\f}\|\na b_{3,\v}\|_2^{\f}\|\na_{H}\db_{\v}\|_2)\|\na\U\|_2^{\f}\|\De\U\|_2^{\f}\,dt\\
\leq&C\v^2\int_0^{t_0}(\|\B_{\v}\|_2^2\|\na \B_{\v}\|_2^2+\|\na\U\|_2^2\|\De \U\|_2^2+\v^2\|b_{3,\v}\|_2^2\|\na b_{3,\v}\|_2^2)\,dt\\
&+\frac{1}{22}\|\na \db_{\v}\|_{L^2(Q_{t_0})}^2+\frac{1}{6}\v^2\|\na B_{3,\v}\|_{L^2(Q_{t_0})}^2\\
\leq&\frac{1}{22}\|\na \db_{\v}\|_{L^2(Q_{t_0})}^2+\frac{1}{6}\v^2\|\na B_{3,\v}\|_{L^2(Q_{t_0})}^2\\
&+C\v^2\Big[(\|\B_0\|_2^2+\v^2\|b_{3,0}\|_2^2)^2+C(\|\U_0\|_{H^1}, L_1, L_2)\Big].
\end{align*}

A similar argument to that for $I_{81}+I_{82}$, yields
\begin{align*}
I_{91}+I_{92}=&-\v^2\int_{Q_{t_0}}(b_{\v}\C\na u_{3,\v} b_3-b_{\v}\C\na u_3b_3)\,dxdydzdt\\
=&-\v^2\int_{Q_{t_0}}(b_{\v}\C\na U_{3,\v})b_3\,dxdydzdt\\
=&-\v^2\int_{Q_{t_0}}\big[\B_{\v}\C\na_H U_{3,\v}-b_{3,\v}\na_H\C\du_{\v}\big]b_3\,dxdydzdt\\
\leq&\v^2\int_0^{t_0}\int_M\Big(\int_{-1}^1|\B_{\v}||\na_H U_{3,\v}|+|b_{3,\v}||\na_H\du_{\v}|\,dz\Big)\Big(\int_{-1}^1|\na_H\B|\,dz\Big)\,dxdydt\\
\leq&C\v^2\int_0^{t_0}(\|\B_{\v}\|_2^{\f}\|\na\B_{\v}\|_2^{\f}\|\na U_{3,\v}\|_2+\|b_{3,\v}\|_2^{\f}\|\na b_{3,\v}\|_2^{\f}\|\na_H \du_{\v}\|_2)\|\na \B\|_2^{\f}\|\De \B\|_2^{\f}\,dt\\
\leq&C\v^2\int_0^{t_0}(\|\B_{\v}\|_2^2\|\na\B_{\v}\|_2^2+\v^2\|b_{3,\v}\|_2^2\|\na b_{3,\v}\|_2^2+\|\na\B\|_2^2\|\De\B\|_2^2)\,dt\\
&+\frac{1}{22}\|\na\du_{\v}\|_{L^2(Q_{t_0})}^2+\frac{1}{6}\v^2\|\na U_{3,\v}\|_{L^2(Q_{t_0})}^2\\
\leq&\frac{1}{16}\|\na\du_{\v}\|_{L^2(Q_{t_0})}^2+\frac{1}{6}\v^2\|\na U_{3,\v}\|_{L^2(Q_{t_0})}^2\\
&+C\v^2\Big[(\|\B_0\|_2^2+\v^2\|b_{3,0}\|_2^2)^2+C(\|\B_0\|_{H^1},L_1,L_2)\Big].
\end{align*}

It is left to estimate the  $I_{10}$ term
\begin{align*}
I_{10}=&\v^2\int_{Q_{t_0}}u_{\v}\C\na b_{3,\v} b_3\,dxdydzdt\\
=&\v^2\int_{Q_{t_0}}u_{\v}\C\na b_{3,\v} b_3\,dxdydzdt-\v^2\int_{Q_{t_0}}u_{\v}\C\na b_3 b_3\,dxdydzdt\\
=&\v^2\int_{Q_{t_0}}u_{\v}\C\na B_{3,\v}b_3\,dxdydzdt\\
=&\v^2\int_{Q_{t_0}}\big[\U_{\v}\C\na_H B_{3,\v}-u_{3,\v}\na_H\C\db_{\v}\big]b_3\,dxdydzdt\\
\leq&\v^2\int_0^{t_0}\int_M\Big(\int_{-1}^1(|\U_{\v}||\na_H B_{3,\v}|+|u_{3,\v}||\na_H\db_{\v}|)\,dz\Big)\Big(\int_{-1}^1|\na_H\B|\,dz\Big)\,dxdydt,
\end{align*}
which is further bounded through the Lemma \ref{b7}, the Poincar\'e and Young inequalities
\begin{align*}
I_{10}\leq&C\v^2\int_0^{t_0}(\|\U_{\v}\|_2^{\f}\|\na\U_{\v}\|_2^{\f}\|\na B_{3,\v}\|_2+\|u_{3,\v}\|_2^{\f}\|\na u_{3,\v}\|_2^{\f}\|\na_H\db_{\v}\|_2)\|\na\B\|_2^{\f}\|\De\B\|_2^{\f}\,dt\\
\leq&C\v^2\int_0^{t_0}(\|\U_{\v}\|_2^2\|\na\U_{\v}\|_2^2+\v^2\|u_{3,\v}\|_2^2\|\na u_{3,\v}\|_2^2+\|\na\B\|_2^2\|\De\B\|_2^2)\,dt\\
&+\frac{1}{6}\v^2\|\na B_{3,\v}\|_{L^2(Q_{t_0})}^2+\frac{1}{22}\|\na_H \db_{\v}\|_{L^2(Q_{t_0})}^2\\
\leq&\frac{1}{6}\v^2\|\na B_{3,\v}\|_{L^2(Q_{t_0})}^2+\frac{1}{22}\|\na\db_{\v}\|_{L^2(Q_{t_0})}^2+C\v^2[(\|\U_0\|_2^2+\v^2\|u_{3,0}\|_2^2)^2+C(\|\B_{0}\|_{H^1},L_1,L_2)].
\end{align*}

In light of  the estimates of  $I_1-I_{10}$ into  (\ref{p4}) yields
\begin{align*}
g(t):=&\big(\|\du_{\v}\|_2^2+\|\db_{\v}\|_2^2+\v^2\|U_{3,\v}\|_2^2+\v^2\|B_{3,\v}\|_2^2\big)(t)\\
&+\int_0^{t}\big(\|\na\du_{\v}\|_2^2+\|\na\db_{\v}\|_2^2+\v^2\|\na U_{3,\v}\|_2^2+\v^2\|\na B_{3,\v}\|_2^2\big)\,ds\\
\leq&C\v^2\Big[(\|\U_0\|_2^2+\|\B_0\|_2^2+\v^2\|u_{3,0}\|_2^2+\v^2\|b_{3,0}\|_2^2+1)^2+C(\|\U_0\|_{H^1}^2,\|\B_0\|_{H^1}^2,L_1,L_2)\Big]\\
&+C\int_0^t(\|\na\U\|_2^2\|\De\U\|_2^2+\|\na\B\|_2^2\|\De\B\|_2^2)(\|\du_{\v}\|_2^2+\|\db_{\v}\|_2^2)=:G(t),
\end{align*}
for a.e. $t\in[0,\infty)$. Therefore, we have
\begin{align*}
G'(t)=&C(\|\na\U\|_2^2\|\De\U\|_2^2+\|\na\B\|_2^2\|\De\B\|_2^2)(\|\du_{\v}\|_2^2+\|\db_{\v}\|_2^2)\\
\leq&C(\|\na\U\|_2^2\|\De\U\|_2^2+\|\na\B\|_2^2\|\De\B\|_2^2)g(t)\\
\leq&C(\|\na\U\|_2^2\|\De\U\|_2^2+\|\na\B\|_2^2\|\De\B\|_2^2)G(t),
\end{align*}
applying  the Gronwall inequality, and using Proposition \ref{5}, we get
\begin{align*}
g(t)\leq &G(t)\leq e^{C\int_0^t(\|\na\U\|_2^2\|\De\U\|_2^2+\|\na\B\|_2^2\|\De\B\|_2^2)\,ds}G(0)\\
\leq&\v^2C(\|\U_0\|_{H^1},\|\B_0\|_{H^1},L_1,L_2)(\|\U_0\|_2^2+\|\B_0\|_2^2+\v^2\|u_{3,0}\|_2^2+\v^2\|b_{3,0}\|_2^2+1)^2,
\end{align*}
where $C$ is a positive constants depending only on $\|\U_0\|_{H^1},\|\B_0\|_{H^1},L_1$ and $L_2$. This completes the proof.
\end{proof}
\vskip .2in

\section{Strong convergence II: the $H^2$ initial data case}
\vskip .1in
 In this section, we deal with  the strong convergence of the  SMHD to the PEM, with the initial data $(\U_0, \B_0) \in H^2(\Om)$, as the aspect ratio parameter $\v$ goes to zero. In order to keep the nation simple, we remove  the subscript index $\v$ of $(\du_{\v}, U_{3,\v}, \db_{\v}, B_{3,\v})$ from the  below, in other words, we replace $(\du_{\v},U_{3,\v},\db_{\v}, B_{3,\v})$ with  $(\du, U_3, \db,B_3)$. We have the following results.
\begin{thm}\label{t2}
Given a periodic function $(\U_0, \B_0)\in H^2(\Om)$, such that
$$\na_H\C\Big(\int_{-1}^1\U_0(x,y,z)\,dz\Big)=0, \qquad \int_{\Om}\U_0(x,y,z)\,dxdydz=0,$$
and
$$\na_H\C\Big(\int_{-1}^1\B_0(x,y,z)\,dz\Big)=0, \qquad \int_{\Om}\B_0(x,y,z)\,dxdydz=0.$$
Let $(\U_{\v}, u_{3,\v}, \B_{\v}, b_{3,\v})$ be the unique local (in time) strong solution to the SMHD and $(\U, u_3, \B, b_3)$ be the unique global strong solution to the PEM, subject to (\ref{pm1})-(\ref{b3}). Denote
$$(\du, U_{3})=(\U_{\v}-\U, u_{3,\v}-u_3),\qquad (\db, B_{3})=(\B_\v-\B, b_{3,\v}-b_3 ).$$

Then, there is a positive constant $\v_0$ depending only on  $\|\U_0\|_{H^2}$, $\|\B_0\|_{H^2}$, $L_1$ and $L_2$, such that, for any $\v\in(0,\v_0)$, the strong solution $(\U_{\v}, u_{3,\v}, \B_{\v}, b_{3,\v})$ of the SMHD exists globally in time, and the following estimate holds
\begin{align*}
&\sup_{0\leq t<\infty}(\|\du\|_{H^1}^2+\v^2\|U_3\|_{H^1}^2+\|\db\|_{H^1}^2+\v^2\|B_3\|_{H^1}^2)(t)\\
&+\int_0^\infty(\|\na\du\|_{H^1}^2+\v^2\|\na U_3\|_{H^1}^2+\|\na\db\|_{H^1}^2+\v^2\|\na B_3\|_{H^1}^2)\,dt\\
\leq &C(\|\U_0\|_{H^2},\|\B_0\|_{H^2},L_1,L_2)\v^2.
\end{align*}
 As a consequence, we have the following strong convergence,
$$(\U_{\v}, \v u_{3,\v}, \B_{\v}, \v b_{3,\v})\rightarrow (\U, 0, \B, 0), ~\text{in}~ L^\infty(0,\infty; H^1(\Om)),$$
$$(\na\U_{\v},\v\na u_{3,\v}, u_{3,\v}, \na\B_{\v}, \v\na b_{3,\v}, b_{3,\v})\rightarrow(\na\U, 0, u_3, \na\B, 0, b_3),~ \text{in}~ L^2(0,\infty; H^1(\Om)),$$
$$(u_{3,\v}, b_{3,\v})\rightarrow (u_3, b_3), ~ \text{in}~ L^\infty(0,\infty; L^2(\Om)),$$
and the convergence rate is of the order $O(\v)$.
\end{thm}
\begin{re}
Theorem \ref{t0} and Theorem  \ref{t2}  deal with  the strong convergence of solutions of SMHD to the  PEM is global and uniform in time, and they converge in the same order. In addition, smoothing the initial data is a strong norm in which convergence occurs.
\end{re}
\begin{re}
Generally, if $(\U_0,\B_0)\in H^k$, with $k\geq 2$, then one can show that
\begin{align*}
&\sup_{0\leq t<\infty}(\|\du\|_{H^{k-1}}^2+\v^2\|U_3\|_{H^{k-1}}^2+\|\db\|_{H^{k-1}}^2+\v^2\|B_3\|_{H^{k-1}}^2)(t)\\
&+\int_0^\infty(\|\na\du\|_{H^{k-1}}^2+\v^2\|\na U_3\|_{H^{k-1}}^2+\|\na\db\|_{H^{k-1}}^2+\v^2\|\na B_3\|_{H^{k-1}}^2)\,dt\\
\leq& C(\|\U_0\|_{H^k},\|\B_0\|_{H^k},L_1,L_2)\v^2,
\end{align*}
 moreover, we have the following strong convergence,
$$(\U_{\v}, \v u_{3,\v}, \B_{\v}, \v b_{3,\v})\rightarrow (\U, 0, \B, 0), ~\text{in}~ L^\infty(0,\infty; H^{k-1}(\Om)),$$
$$(\na\U_{\v},\v\na u_{3,\v}, u_{3,\v}, \na\B_{\v}, \v\na b_{3,\v}, b_{3,\v})\rightarrow(\na\U, 0, u_3, \na\B, 0, b_3),~ \text{in}~ L^2(0,\infty; H^{k-1}(\Om)),$$
$$(u_{3,\v}, b_{3,\v})\rightarrow (u_3, b_3), ~ \text{in}~ L^\infty(0,\infty; H^{k-2}(\Om)),$$
and the convergence rate is of the order $O(\v)$.

This can be achieved by making higher energy estimates for the difference system (\ref{n1}).
\end{re}

\begin{re}
The smoothing effect is observed in the SMHD and the PEM to the unique strong solutions, we can also prove that in this theorem (rather than theorem \ref{t0}), the strong convergence to a stronger norm always starts from the initial time, in particular, $(\U_{\v}, u_{3,\v})\rightarrow (\U, u_3)$ and $(\B_{\v},b_{3,\v})\rightarrow(\B, b_3)$ in $C^k(\overline\Om\times(T,\infty)),$  for any given  time $T>0$ and  integer $k\geq 0$.
\end{re}

In this section,  we give the proof of Theorem \ref{t2}. Let $(\U_0, \B_0)\in H^2(\Om)$, and assume that
\begin{align}\label{m1}
\na_H\C\Big(\int_{-1}^1\U_0(x,y,z)\,dz\Big)=0 \quad\text{and}\quad \na_H\C\Big(\int_{-1}^1\B_0(x,y,z)\,dz\Big)=0,
\end{align}
for all $(x,y)\in M$. Set $u_0=(\U_0, u_{3,0})$, $b_0=(\B_0,b_{3,0})$, $\na\C u_0=0$ and  $\na\C b_0=0$, with $u_{3,0}$ and $b_{3,0}$ given by (\ref{t3}) and (\ref{mm}), then $(u_0, b_0)\in H^1(\Om)$. Through the same argument as  the standard MHD equations, see, e.g.,\cite{DL}, it can be  proved that there is a unique local (in time) strong solution $(\U_{\v}, u_{3,\v}, \B_{\v}, b_{3,\v})$ to the SMHD subject to (\ref{pm1})-(\ref{b3}). $T_{\v}^*$ represents the maximal existence time of the strong solution. Because of the smoothing effect of SMHD on the unique strong solution, it can be proved that  strong solution $(\U_{\v},u_{3,\v}, \B_{\v}, b_{3,\v})$ is smooth over the time interval $(0,T_{\v}^*)$, recalling that $(\U, u_3, \B, b_3)$ is smooth away from the initial time, the same as $(\du,U_{3}, \db, B_{3})$. This guarantees the validity of the argument in the following proof.

We need to make  \emph{a priori} estimates of $(\du_{\v}, U_{3,\v}, \db_{\v},B_{3,\v})$. We begin with  the basic energy estimate described in the following proposition.
\vskip .2in

\begin{prop}\label{m7}(Basic $L^2$ energy estimate). The following basic energy estimate holds
\begin{align*}
&\sup_{0\leq s\leq t}(\|\du\|_2^2+\v^2\|U_{3}\|_2^2+\|\db\|_2^2+\v^2\|B_{3}\|_2^2)\\
&+\int_0^t(\|\na \du\|_2^2+\v^2\|\na U_{3}\|_2^2+\|\na \db\|_2^2+\v^2\|\na B_{3}\|_2^2)\,ds\\
\leq&C\v^2(\|\U_0\|_2^2+\|\B_0\|_2^2+\v^2\|u_{3,0}\|_2^2+\v^2\|b_{3,0}\|_2^2+1)^2,
\end{align*}
for any $t\in[0,T_{\v}^*)$, where $C$ is a constant depending on $\|\U\|_{H^1}$, $\|\B\|_{H^1}$, $L_1$ and $L_2$.
\end{prop}

\begin{proof}
This is a direct consequence of Proposition \ref{d2}.
\end{proof}
\vskip.1in
The following proposition is the first order energy estimate.
\begin{prop}\label{m8}($H^1$ energy estimates)
There is a constant $\delta_0>0$ depending only on $L_1$ and $L_2$, so that, we have the following estimate
\begin{align*}
&\sup_{0\leq s\leq t}(\|\na\du\|_2^2+\v^2\|\na U_3\|_2^2+\|\na\db\|_2^2+\v^2\|\na B_3\|_2^2)\\
&+\int_0^t(\|\De\du\|_2^2+\v^2\|\De U_3\|_2^2+\|\De\db\|_2^2+\v^2\|\De B_3\|_2^2)\,ds\\
\leq &2C_1\v^2e^{C_1(1+\v^4)\int_0^t(\|\De\U\|_2^2\|\na\De\U\|_2^2+\|\De\B\|_2^2\|\na\De\B\|_2^2)\,ds}\\
&\times\int_0^t(1+\|\De\U\|_2^2+\|\De\B\|_2^2)(\na\p_t\U\|_2^2+\|\na\p_t\B\|_2^2+\|\na\De\U\|_2^2+\|\na\De\B\|_2^2)\,ds.
\end{align*}
for any $t\in[0,T_{\v}^*)$, provided that
\begin{align*}
&\sup_{0\leq s\leq t}(\|\na\du\|_2^2+\v^2\|\na U_3\|_2^2+\|\na\db\|_2^2+\v^2\|\na B_3\|_2^2)\leq \delta_0^2,
\end{align*}
where $C$ is a positive constant depending only on $L_1$ and $L_2$.
\end{prop}
\begin{proof}[Proof]
Multiplying the  equations $(\ref{n1})_1$, $(\ref{n1})_2$, $(\ref{n1})_3$ and $(\ref{n1})_4$ by $-\De\du$, $-\De U_3$, $-\De\db$ and  $-\De B_3$, respectively, integrating the result over $\Om$, then it follows from integration by parts that
\begin{align}\label{t28}
&\f\ddt(\|\na \du\|_2^2+\v^2\|\na U_{3}\|_2^2+\|\na \db\|_2^2
+\v^2\|\na B_3\|_2^2)\notag\\
&+\int_{\Om}\|\De\du\|_2^2+\v^2\|\De U_3\|_2^2+\|\De\db\|_2^2+\v^2\|\De B_3\|_2^2\,dxdydz\notag\\
=&\int_{\Om}\big[(U\C\na)\du+(u\C\na)\du+(U\C\na)\U\big]\C\De\du\,dxdydz\notag\\
&-\int_{\Om}\big[(B\C\na)\db+(b\C\na)\db+(B\C\na)\B\big]\C\De\du\,dxdydz\notag\\
&+\v^2\int_{\Om}\big(U\C\na U_3+u\C\na U_3+U\C\na u_3\big)\De U_3\,dxdydz\notag\\
&-\v^2\int_{\Om}\big(B\C\na B_3+b\C\na B_3+B\C\na b_3\big)\De U_3\,dxdydz\notag\\
&+\v^2\int_{\Om}\big(\p_t u_3+u\C\na u_3-\De u_3-b\C\na b_3\big)\De U_3\,dxdydz\notag\\
&+\int_{\Om}\big[(U\C\na)\db+(u\C\na)\db+(U\C\na)\B\big]\C\De\db\,dxdydz\notag\\
&-\int_{\Om}\big[(B\C\na)\du+(b\C\na)\du+(B\C\na)\U\big]\C\De\db\,dxdydz\notag\\
&+\v^2\int_{\Om}\big[U\C\na B_3+u\C\na B_3+U\C\na b_3\big]\De B_3\,dxdydz\notag\\
&-\v^2\int_{\Om}\big[B\C\na U_3+b\C\na U_3+B\C\na u_3\big]\De B_3\,dxdydz\notag\\
&+\v^2\int_{\Om}\big(\p_t b_3+u\C\na b_3-\De b_3-b\C\na u_3\big)\De B_3\,dxdydz\notag\\
:=&J_1+J_2+J_3+J_4+J_5+J_6+J_7+J_8+J_9+J_{10}.
\end{align}

Let us estimate the ten terms appearing above. First, by Lemma \ref{b8},  it follows from the Young and Poincar\'e inequalities that
\begin{align*}
J_1=&\int_{\Om}\big[(U\C\na)\du+(u\C\na)\du+(U\C\na)\U\big]\C\De\du\,dxdydz\\
\leq&C(\|\na \du\|_2\|\De\du\|_2+\|\na\U\|_2^{\f}\|\De\U\|_2^{\f}\|\na U\|_2^{\f}\|\De U\|_2^{\f})\|\De\du\|_2\\
\leq&\frac{1}{15}\|\De\du\|_2^2+C(\|\na \du\|_2^2\|\De\du\|_2^2+\|\na\U\|_2^2\|\De\U\|_2^2\|\na\du\|_2^2)\\
\leq&\frac{1}{15}\|\De\du\|_2^2+C(\|\na\du\|_2^2\|\De\du\|_2^2+\|\De\U\|_2^2\|\na \De\U\|_2^2\|\na \du\|_2^2).
\end{align*}

A similar argument to that for $J_1$, yields
\begin{align*}
J_2=&\int_{\Om}\big[(B\C\na)\db+(b\C\na)\db+(B\C\na)\B\big]\C\De\du\,dxdydz\\
\leq&C(\|\na \db\|_2\|\De\db\|_2+\|\na\B\|_2^{\f}\|\De\B\|_2^{\f}\|\na \db\|_2^{\f}\|\De\db\|_2^{\f})\|\De\du\|_2\\
\leq&\frac{1}{15}\|\De\du\|_2^2+\frac{1}{10}\|\De\db\|_2^2+C(\|\na\db\|_2^2\|\De\db\|_2^2
+\|\De\B\|_2^2\|\na\De\B\|_2^2\|\na\db\|_2^2).
\end{align*}

Using the fact that
$$\|\na u_3\|_2\leq C\|\De \U\|_2, \qquad \|\De u_3\|_2\leq C\|\na \De \U\|_2,$$
which can be easily verified by recalling $u_3(x,y,z,t)=-\int_0^z\na_H\C\U(x,y, \xi,t)\,d\xi$ and using the Poincar\'e inequality, we get
\begin{align*}
J_3=&\v^2\int_{\Om}\big[(U\C\na)U_3+(u\C\na) U_3+(U\C\na) u_3\big]\De U_3\,dxdydz\\
\leq&C\v^2\Big[(\|\na \du\|_2^{\f}\|\De\du\|_2^{\f}+\|\na\U\|_2^{\f}\|\De\U\|_2^{\f})\|\na U_3\|_2^{\f}\|\De U_3\|_2^{\f}\\
&+\|\na \du\|_2^{\f}\|\De\du\|_2^{\f}\|\na u_3\|_2^{\f}\|\De u_3\|_2^{\f}\Big]\|\De U_3\|_2\\
\leq&\frac{1}{15}\|\De\du\|_2^2+\frac{1}{10}\v^2\|\De U_3\|_2^2+C(\|\na \du\|_2^2\|\De \du\|_2^2+\v^4\|\na U_3\|_2^2\|\De U_3\|_2^2)\\
&+C\v^2\|\na \U\|_2^2\|\De\U\|_2^2\|\na U_3\|_2^2+C\v^4\|\na u_3\|_2^2\|\De u_3\|_2^2\|\na \du\|_2^2\\
\leq&\frac{1}{15}\|\De\du\|_2^2+\frac{1}{10}\v^2\|\De U_3\|_2^2+C(\|\na \du\|_2^2+\v^2\|\na U_3\|_2^2)\\
&\times\big[\|\De\du\|_2^2+\v^2\|\De U_3\|_2^2+(1+\v^4)\|\De\U\|_2^2\|\na\De\U\|_2^2\big].
\end{align*}

Using again Lemma \ref{b8} and H\"{o}lder, Poincar\'e and  Young inequalities give that
\begin{align*}
J_4=&\v^2\int_{\Om}(B\C\na B_3+b\C\na B_3+B\C\na b_3)\De U_3\,dxdydz\\
\leq&C\v^2\Big[(\|\na \db\|_2^{\f}\|\De \db\|_2^{\f}+\|\na \B\|_2^{\f}\|\De\B\|_2^{\f})\|\na B_3\|_2^{\f}\|\De B_3\|_2^{\f}\\
&+\|\na\db\|_2^{\f}\|\De\db\|_2^{\f}\|\na b_3\|_2^{\f}\|\De b_3\|_2^{\f}\Big]\|\De U_3\|_2\\
\leq&\frac{1}{10}\v^2\|\De U_3\|_2^2+\frac{1}{10}\|\De\db\|_2^2+\frac{1}{10}\v^2\|\De B_3\|_2^2+C(\|\na\db\|_2^2\|\De\db\|_2^2+\v^4\|\na B_3\|_2^2\|\De B_3\|_2^2)\\
&+C\v^2\|\na \B\|_2^2\|\De\B\|_2^2\|\na B_3\|_2^2+C\v^4\|\na b_3\|_2^2\|\De b_3\|_2^2\|\na \db\|_2^2\\
\leq&\frac{1}{10}\v^2\|\De U_3\|_2^2+\frac{1}{10}\|\De\db\|_2^2+\frac{1}{10}\v^2\|\De B_3\|_2^2\\
&+C(\|\na \db\|_2^2+\v^2\|\na B_3\|_2^2)(\|\De\db\|_2^2+\v^2\|\De B_3\|_2^2+(1+\v^4)\|\De\B\|_2^2\|\na\De\B\|_2^2),
\end{align*}
where in the last step, we have used the fact that
$$\|\na b_3\|_2\leq C\|\De \B\|_2, \qquad \|\De b_3\|_2\leq C\|\na \De \B\|_2.$$

For $J_5$, we have
\begin{align*}
J_5=&\v^2\int_{\Om}(\p_t u_3+u\C\na u_3-\De u_3-b\C\na b_3)\De U_3\,dxdydz\\
\leq&\v^2(\|\p_t u_3\|_2+\|\De u_3\|_2)\|\De U_3\|_2+C\v^2\|\na\U\|_2^{\f}\|\De\U\|_2^{\f}\|\na u_3\|_2^{\f}\|\De u_3\|_2^{\f}\|\De U_3\|_2\\
&+C\v^2\|\na \B\|_2^{\f}\|\De\B\|_2^{\f}\|\na b_3\|_2^{\f}\|\De b_3\|_2^{\f}\|\De U_3\|_2\\
\leq&\frac{1}{10}\v^2\|\De U_3\|_2^2+C\v^2(\|\p_t u_3\|_2^2+\|\De u_3\|_2^2+\|\na \U\|_2^2\|\De\U\|_2^2+\|\na u_3\|_2^2\|\De u_3\|_2^2\\
&+\|\na \B\|_2^2\|\De\B\|_2^2+\|\na b_3\|_2^2\|\De b_3\|_2^2)\\
\leq&\frac{1}{10}\v^2\|\De U_3\|_2^2+C\v^2(\|\na\p_t \U\|_2^2+\|\na \De\U\|_2^2+\|\De\U\|_2^2\|\na\De\U\|_2^2+\|\De\B\|_2^2\|\na\De\B\|_2^2).
\end{align*}

By the H\"{o}lder and  Young inequalities, we deduce
\begin{align*}
J_6=&\int_{\Om}\big[(U\C\na)\db+(u\C\na)\db+(U\C\na)\B\big]\C\De\db\,dxdydz\\
\leq&C\big(\|\na \du\|_2^{\f}\|\De\du\|_2^{\f}\|\na\db\|_2^{\f}\|\De\db\|_2^{\f}+\|\na\U\|_2^{\f}\|\De\U\|_2^{\f}\|\na\db\|_2^{\f}\|\De\db\|_2^{\f}\\
&+\|\na\du\|_2^{\f}\|\De\du\|_2^{\f}\|\na\B\|_2^{\f}\|\De\B\|_2^{\f}\big)\|\De\db\|_2\\
\leq&\frac{1}{10}\|\De\db\|_2^2+\frac{1}{15}\|\De\du\|_2^2+C(\|\na \du\|_2^2\|\De\du\|_2^2+\|\na\db\|_2^2\|\De\db\|_2^2\\
&+\|\na\U\|_2^2\|\De\U\|_2^2\|\na \db\|_2^2+\|\na \B\|_2^2\|\De\B\|_2^2\|\na\du\|_2^2)\\
\leq&\frac{1}{10}\|\De\db\|_2^2+\frac{1}{15}\|\De\du\|_2^2+C(\|\na\du\|_2^2 \|\De\du\|_2^2+\|\na\db\|_2^2\|\De\db\|_2^2\\
&+\|\De\U\|_2^2\|\na\De\U\|_2^2\|\na\db\|_2^2+\|\De\B\|_2^2\|\na\De\B\|_2^2\|\na\du\|_2^2).
\end{align*}

We can estimate $J_7$ as follows
\begin{align*}
J_7=&\int_{\Om}\big[(B\C\na)\du+(b\C\na)\du+(B\C\na)\U\big]\C\De\db\,dxdydz\\
=&\int_{\Om}\big(\|\na\db\|_2^{\f}\|\De\db\|_2^{\f}\|\na\du\|_2^{\f}\|\De\du\|_2^{\f}\\
&+\|\na\B\|_2^{\f}\|\De\B\|_2^{\f}\|\na\du\|_2^{\f}\|\De\du\|_2^{\f}+\|\na\db\|_2^{\f}\|\De\db\|_2^{\f}\|\na \U\|_2^{\f}\|\De\U\|_2^{\f}\big)\|\De\db\|_2\\
\leq&\frac{1}{10}\|\De\db\|_2^2+\frac{1}{15}\|\De\du\|_2^2+C(\|\na\db\|_2^2\|\De\db\|_2^2+\|\na\du\|_2^2\|\De\du\|_2^2\\
&+\|\na\B\|_2^2\|\De\B\|_2^2\|\na \du\|_2^2+\|\na\U\|_2^2\|\De\U\|_2^2\|\na\db\|_2^2)\\
\leq&\frac{1}{10}\|\De\db\|_2^2+\frac{1}{15}\|\De\du\|_2^2+C(\|\na\db\|_2^2\|\De\db\|_2^2+\|\na\du\|_2^2\|\De\du\|_2^2\\
&+\|\De\B\|_2^2\|\na\De\B\|_2^2\|\na\du\|_2^2+\|\De\U\|_2^2\|\na\De\U\|_2^2\|\na\db\|_2^2).
\end{align*}

Applying the Lemma \ref{b8} and Young inequalities once again, one has
\begin{align*}
J_8=&\v^2\int_{\Om}[(U\C\na)B_3+(u\C\na)B_3+(U\C\na)b_3]\De B_3\,dxdydz\\
\leq&C\v^2\Big[(\|\na \du\|_2^{\f}\|\De\du\|_2^{\f}+\|\na\U\|_2^{\f}\|\De\U\|_2^{\f})\|\na B_3\|_2^{\f}\|\De B_3\|_2^{\f}\\
&+\|\na \du\|_2^{\f}\|\De\du\|_2^{\f}\|\na b_3\|_2^{\f}\|\De b_3\|_2^{\f}\Big]\|\De B_3\|_2\\
\leq&\frac{1}{10}\v^2\|\De B_3\|_2^2+\frac{1}{15}\|\De\du\|_2^2+C(\|\na\du\|_2^2+\v^2\|\na B_3\|_2^2)\\
&\times(\|\De\du\|_2^2+\v^2\|\De B_3\|_2^2+\|\De\U\|_2^2\|\na\De\U\|_2^2+\v^4\|\De\B\|_2^2\|\na\De\B\|_2^2),
\end{align*}
moreover
\begin{align*}
J_9=&\v^2\int_{\Om}\big[(B\C\na)U_3+b\C\na U_3+B\C\na u_3\big]\De B_3\,dxdydz\\
\leq&C\v^2\Big[(\|\na\db\|_2^{\f}\|\De\db\|_2^{\f}+\|\na\B\|_2^{\f}\|\De\B\|_2^{\f})\|\na U_3\|_2^{\f}\|\De U_3\|_2^{\f}\\
+&\|\na\db\|_2^{\f}\|\De\db\|_2^{\f}\|\na u_3\|_2^{\f}\|\De u_3\|_2^{\f}\Big]\|\De B_3\|_2\\
\leq&\frac{1}{10}\v^2\|\De B_3\|_2^2+\frac{1}{10}\v^2\|\De U_3\|_2^2+C(\|\na\db\|_2^2\|\De\db\|_2^2+\v^4\|\na U_3\|_2^2\|\De U_3\|_2^2\\
&+\v^2\|\na\B\|_2^2\|\De\B\|_2^2\|\na U_3\|_2^2+C\v^4\|\na u_3\|_2^2\|\De u_3\|_2^2\|\na \db\|_2^2)\\
\leq&\frac{1}{10}\v^2\|\De B_3\|_2^2+\frac{1}{10}\v^2\|\De U_3\|_2^2+C(\|\na\db\|_2^2+\v^2\|\na U_3\|_2^2)\\
&\times(\|\De\db\|_2^2+\v^2\|\De U_3\|_2^2+\|\De\B\|_2^2\|\na\De\B\|_2^2+\v^4\|\De\U\|_2^2\|\na\De\U\|_2^2).
\end{align*}

For the last term $J_{10}$, we get that
\begin{align*}
J_{10}=&\v^2\int_{\Om}(\p_t b_3+u\C\na b_3-\De b_3-b\C\na u_3)\De B_3\,dxdydz\\
\leq&\v^2(\|\p_t b_3\|_2+\|\De b_3\|_2)\|\De B_3\|_2+C\v^2\|\na\U\|_2^{\f}\|\De\U\|_2^{\f}\|\na b_3\|_2^{\f}\|\De b_3\|_2^{\f}\|\De B_3\|_2\\
+&C\v^2\|\na\B\|_2^{\f}\|\De\B\|_2^{\f}\|\na u_3\|_2^{\f}\|\De u_3\|_2^{\f}\|\De B_3\|_2\\
\leq&\frac{1}{10}\v^2\|\De B_3\|_2^2+C\v^2\big(\|\p_t b_3\|_2^2+\|\De b_3\|_2^2+\|\na\U\|_2^2\|\De\U\|_2^2\\
+&\|\na b_3\|_2^2\|\De b_3\|_2^2+\|\na \B\|_2^2\|\De\B\|_2^2+\|\na u_3\|_2^2\|\De u_3\|_2^2\big)\\
\leq&\frac{1}{10}\v^2\|\De B_3\|_2^2+C\v^2\big(\|\na\p_t\B\|_2^2+\|\na\De\B\|_2^2+\|\De\U\|_2^2\|\na\De\U\|_2^2+
\|\De\B\|_2^2\|\na\De\B\|_2^2\big).
\end{align*}

 In view of the estimates of $J_1-J_{10}$, we derive from (\ref{t28}) the differential inequality
\begin{align*}
&\f\ddt\big(\|\na\du\|_2^2+\v^2\|\na U_3\|_2^2+\|\na \db\|_2^2+\v^2\|\na B_3\|_2^2\big)+\frac{3}{5}\big(\|\De\du\|_2^2+\v^2\|\De U_3\|_2^2+\|\De\db\|_2^2+\v^2\|\De B_3\|_2^2\big)\\
\leq&C_1\big(\|\na\du\|_2^2+\v^2\|\na U_3\|_2^2+\|\na\db\|_2^2+\v^2\|\na B_3\|_2^2\big)\big[\|\De\du\|_2^2+\v^2\|\De U_3\|_2^2+\|\De\db\|_2^2+\v^2\|\De B_3\|_2^2\\
&+(1+\v^4)\|\De\U\|_2^2\|\na\De\U\|_2^2+(1+\v^4)\|\De\B\|_2^2\|\na\De\B\|_2^2\big]\\
&+C_1\v^2(1+\|\De\U\|_2^2+\|\De\B\|_2^2)(\|\na\p_t\U\|_2^2+\|\na\p_t\B\|_2^2+\|\na\De\U\|_2^2+\|\na\De\B\|_2^2).
\end{align*}

By the assumption $\sup_{0\leq s\leq t}\big(\|\na\du\|_2^2+\v^2\|\na U_3\|_2^2+\|\na\db\|_2^2+\v^2\|\De B_3\|_2^2\big)\leq \delta_0^2$, choosing $\delta_0=\sqrt{\frac{1}{10C_1}}$, it follows from the above inequality that
\begin{align*}
&\ddt\big(\|\na \du\|_2^2+\v^2\|\na U_3\|_2^2+\|\na\db\|_2^2+\v^2\|\na B_3\|_2^2\big)+\|\De\du\|_2^2+\v^2\|\De U_3\|_2^2+\|\De\db\|_2^2
+\v^2\|\De B_3\|_2^2\\
\leq&2C_1(1+\v^4)(\|\De \U\|_2^2\|\na\De\U\|_2^2+\|\De\B\|_2^2\|\na\De\B\|_2^2)(\|\na \du\|_2^2+\v^2\|\na U_3\|_2^2+\|\na\db\|_2^2+\v^2\|\na B_3\|_2^2)\\
&+2C_1\v^2(1+\|\De\U\|_2^2+\|\De\B\|_2^2)(\|\na \p_t\U\|_2^2+\|\na\p_t\B\|_2^2+\|\na\De\U\|_2^2+\|\na \De\B\|_2^2),
\end{align*}
recalling  $(\du, U_3)|_{t=0}=0$ and  $(\db,B_3)|_{t=0}=0$, it follows from the Gronwall inequality that
\begin{align*}
&\sup_{0\leq s\leq t}(\|\du\|_{H^1}^2+\v^2\|U_3\|_{H^1}^2+\|\db\|_{H^1}^2+\v^2\|B_3\|_{H^1}^2)\\
&+\int_0^t(\|\De\du\|_2^2+\v^2\|\De U_3\|_2^2+\|\De\db\|_2^2+\v^2\|\De B_3\|_2^2)\,ds\\
\leq &2C_1\v^2e^{2C_1(1+\v^4)\int_0^t(\|\De\U\|_2^2\|\na\De\U\|_2^2+\|\De\B\|_2^2\|\na\De\B\|_2^2\,ds)}\\
&\times\int_0^t(1+\|\De\U\|_2^2+\|\De\B\|_2^2)(\na\p_t\U\|_2^2+\|\na\p_t\B\|_2^2+\|\na\De\U\|_2^2+\|\na\De\B\|_2^2)\,ds,
\end{align*}
proving the conclusion.
\end{proof}

Thanks to  Propositions \ref{m7}-\ref{m8} and  Proposition \ref{t1}, we can prove the following results.
\vskip .2in
\begin{prop}\label{m9}
There is a positive constant $\v_0$ that only depends on $\|\U_0\|_{H^2}$, $\|\B_0\|_{H^2}$, $L_1$ and $L_2$, so that for any $\v\in(0,\v_0)$, there is a unique global strong solution $(u_{\v}, b_{\v})$  to the SMHD, subject to (\ref{pm1})-(\ref{b3}). In addition, we have the following  estimate
\begin{align*}
&\sup_{0\leq s< \infty}(\|\du\|_{H^1}^2+\v^2\|U_3\|_{H^1}^2+\|\db\|_{H^1}^2+\v^2\|B_3\|_{H^1}^2)\\
&+\int_0^\infty(\|\na\du\|_{H^1}^2+\v^2\|\na U_3\|_{H^1}^2+\|\na\db\|_{H^1}^1+\v^2\|\na B_3\|_{H^1}^2)\,dt\\
&\leq C\v^2,
\end{align*}
where $C$ is a positive constant depending only on $\|\U_0\|_{H^2}, \|\B_0\|_{H^2},L_1$ and $L_2$.
\end{prop}
\begin{proof}
Denote $T^*_{\v}$ be the maximal existence time of the strong solutions $(\U_{\v},u_{3,\v}, \B_{\v},b_{3,\v})$ to the SMHD,  subject to the  conditions (\ref{pm1})-(\ref{b3}). According to  Proposition \ref{5} and Proposition \ref{m7}, we get the estimate
\begin{align}\label{m10}
&\sup_{0\leq s<T_{\v}^*}(\|\du\|_2^2+\v^2\|U_{3}\|_2^2+\|\db\|_2^2+\v^2\|B_{3}\|_2^2)\notag\\
&+\int_0^{T_{\v}^*} (\|\na \du\|_2^2+\v^2\|\na U_{3}\|_2^2+\|\na \db\|_2^2+\v^2\|\na B_{3}\|_2^2)\,ds\notag\\
\leq& R_1\v^2,
\end{align}
where $R_1$ is a positive constant depending only on $\|\U_0\|_{H^1}, \|\B_0\|_{H^1}, L_1$ and $L_2$. Let $\delta_0$ be a constant depending only on  $L_1$ and $L_2$ in Proposition \ref{m8}. We define
$$t_{\v}^*:=\sup\Big\{t\in (0, T_{\v}^*)\Big|\sup_{0\leq s\leq t}(\|\na \du\|_2^2+\|\na\db\|_2^2+\v^2\|\na U_{3}\|_2^2+\v^2\|\na B_{3}\|_2^2)\leq\delta_0^2\Big\}.$$

On the basis of  Proposition \ref{m8} and Proposition \ref{t1}, for any $t\in[0,t_{\v}^*)$, the following estimate holds
\begin{align}\label{m11}
&\sup_{0\leq s\leq t}(\|\na \du\|_2^2+\|\na \db\|_2^2+\v^2\|\na U_{3}\|_2^2+\v^2\|\na B_{3}\|_2^2)\notag\\
&+\int_0^t(\|\De\du\|_2^2+\v^2\|\De U_{3}\|_2^2+\|\De\db\|_2^2+\v^2\|\De B_{3}\|_2^2)\,ds\notag\\
\leq&R_2\v^2,
\end{align}
where $R_2$ is a positive constant depending only on $\|\U_0\|_{H^2},|\B_0\|_{H^2}, L_1$ and $L_2$. Setting $\v_0=\delta_0\sqrt{\frac{1}{2R_2}}$, for any $\v\in(0,\v_0)$,  so the inequality above implies that
\begin{align*}
&\sup_{0\leq s\leq t}(\|\na \du\|_2^2+\|\na\db\|_2^2+\v^2\|\na U_{3}\|_2^2+\v^2\|\na B_{3}\|_2^2)\\
&+\int_0^t(\|\De\du\|_2^2+\v^2\|\De U_{3}\|_2^2+\|\De\db\|_2^2+\v^2\|\De B_{3}\|_2^2)\,ds\\
\leq&\frac{\delta_0^2}{2},
\end{align*}
and for any $t_0\in[0, t_{\v}^*)$, especially when it gives
\begin{align*}
\sup_{0\leq s< t_{\v}^*}(\|\na \du\|_2^2+\|\na\db\|_2^2+\v^2\|\na U_{3}\|_2^2+\v^2\|\na B_{3}\|_2^2)
\leq\frac{\delta_0^2}{2}.
\end{align*}

So that, according to the definition of $t_{\v}^*$, we must have $t_{\v}^*=T_{\v}^*$. On account of  this, it is obvious that (\ref{m11}) holds for any $t\in[0,T_{\v}^*)$.

We assert that it must have $T_{\v}^*=\infty$. Assume in contradiction that  $T_{\v}^*<\infty$, then, recalling that (\ref{m11}) is true for any $t\in[0,T_{\v}^*)$, by the local well-posedness result of the SMHD, we can extend the strong solution $(u_{\v}, b_{\v})$ beyond $T_{\v}^*$, which contradicts to the  definition of $T_{\v}^*$. Therefore, combining (\ref{m10}) and  (\ref{m11}) to draw the  conclusion.
\end{proof}
\vskip .2in
On account of  Proposition \ref{m9}, one can  give  the proof of Theorem \ref{t2} in the  following.
\vskip .2in
\begin{proof}[Proof of Theorem \ref{t2}]
Let $\v_0$ depend only on $\|\U_0\|_{H^2}$, $\|\B_0\|_{H^2}$, $L_1$ and $L_2$ in Proposition \ref{m9}. According to Proposition \ref{m9}, for any $\v\in (0,\v_0)$, there is a unique global strong solution $(u_{\v}, b_{\v})$ to the SMHD, subject to the conditions (\ref{pm1})-(\ref{b3}). In addition, we have the  following estimate
 \begin{align*}
&\sup_{0\leq s<\infty}(\|\du\|_{H^1}^2+\v^2\|U_{3}\|_{H^1}^2+\|\db\|_{H^1}^2+\v^2\|B_{3}\|_{H^1}^2)(t)\\
&+\int_0^\infty(\|\na\du\|_{H^1}^2+\v^2\|\na U_{3}\|_{H^1}^2+\|\na\db\|_{H^1}^2+\v^2\|\na B_{3}\|_{H^1}^2)\\
\leq&C(\|\U_0\|_{H^2}, \|\B_0\|_{H^2}, L_1, L_2)\v^2,
\end{align*}
where $(\du,U_{3})=(\U_{\v},u_{3,\v})-(\U,u_3)$ and  $(\db,B_{3})=(\B_{\v},b_{3,\v})-(\B,b_3)$. This proves the estimate stated in the theorem, of which the strong convergence is only a direct  deduction of this estimate. This completes the proof of Theorem \ref{t2}.
\end{proof}

\vskip .2in
\centerline{\bf Acknowledgments}
\vskip .1in
This work is supported by the National Natural Science Foundation of China grant 11971331, 12125102, and Sichuan Youth  Science and Technology Foundation 2021 JDTD0024.
\vskip .2in

\centerline{\bf Declarations}
{\bf Competing interests} The authors have no relevant financial or non-financial interests to disclose.
\vskip .1in
{\bf Funding} The work of L. Du and  D. Li is funded by the National Natural Science Foundation of China (grant 11971331, 12125102), and Sichuan Youth  Science and Technology Foundation (2021 JDTD0024).
\vskip .1in
{\bf Data availability statement} Our manuscript has no associated data.

\end{document}